\documentclass[12pt, a4paper]{amsart}
\usepackage[a4paper, top=1in, bottom=1in, margin=1in]{geometry} 
\usepackage{amssymb}
\usepackage{xspace}
\newcommand{\nn}{\mathbb{N}}
\newcommand{\zz}{\mathbb{Z}}

\newcommand {\C}{$C^{*}$-algebra\xspace}
\newcommand {\Cs}{$C^{*}$-algebras\xspace}

\newcommand{\XX}{\mathfrak{X}}
\newcommand{\YY}{\mathcal{Y}}
\newcommand{\torus}{\mathbb{T}}
\newcommand{\Tspace}{\mathcal{T}}

\newcommand{\rg}{\mathcal{R}}

\newcommand{\ruhf}{\rg}
\newcommand{\Truhf}[1]{\Tspace\!\ast\!\ruhf(#1)}

\newcommand{\gcheck}{\check{G}}
\newcommand{\gcheckinf}{\check{G}_\infty}
\newcommand{\gchecking}{\gcheckinf}
\newcommand{\RTg}{\mathcal{G}}
\newcommand{\zcheck}{\check{Z}}
\newcommand{\tg}{Z}
\newcommand{\TG}{Z}
\newcommand{\Ind}{\mathop{\mathrm{Ind}}\nolimits}

\newcommand{\supp}{\mathop{\mathrm{supp}}}

\newcommand{\pipull}[1]{\pi\!^*\!{#1}}

\newcommand{\brauer}{\mathop{\rm Brauer}}
\newcommand{\torussheaf}{\underline{\torus}}

\newcommand {\etale}{\'etale\xspace}

\newtheorem{theorem}{Theorem}[section]
\newtheorem{prop}[theorem]{Proposition}

\newtheorem{lemma}[theorem]{Lemma}
\newtheorem{remark}[theorem]{Remark}

\theoremstyle{definition}
\newtheorem{mydef}[theorem]{Definition}
\newtheorem{example}[theorem]{Example}
\numberwithin{equation}{theorem}

\begin{document}
\title{Limits of Groupoid C*-Algebras Arising from Open Covers}
\author{A. Censor}
\address{Aviv Censor, Department of Mathematics,
University of California at Riverside, Riverside, CA, 92521, U.S.A.}
\email{avivc@math.ucr.edu}
\author{D. Markiewicz}
\address{Daniel Markiewicz, Department of Mathematics,
Ben-Gurion University of the Negev, P.O.B. 653, Beersheva 84105,
Israel.} \email{danielm@math.bgu.ac.il}
\date{\today}
\begin{abstract}
I. Raeburn and J. Taylor have constructed continuous-trace \Cs with
a prescribed Dixmier-Douady class, which also depend on the choice
of an open cover of the spectrum. We study the asymptotic behavior
of these algebras with respect to certain refinements of the cover
and appropriate extension of cocycles. This leads to the analysis of
a limit groupoid $G$ and a cocycle $\sigma$, and the algebra $C^*(G,
\sigma)$ may be regarded as a generalized direct limit of the
Raeburn-Taylor algebras. As a special case, all UHF \Cs arise from
this limit construction.
\end{abstract}
\maketitle

\section{Introduction}

In their study of continuous-trace \Cs, J. Dixmier and A.
Douady~\cite{dixmier-douady}, \cite{dixmier} introduced a complete
invariant classifying these algebras up to spectrum preserving
Morita equivalence. If $A$ is a conti\-nuous-trace \C with spectrum
$\Tspace$, then its Dixmier-Douady invariant $\delta(A)$ is an
element of the sheaf cohomology group $H^2(\Tspace,\torussheaf)$,
where $\torussheaf$ denotes the sheaf of germs of continuous
circle-valued functions on $\Tspace$. We remark that
$H^2(\Tspace,\torussheaf)$ is naturally isomorphic to the more
familiar cohomology group $H^3(\Tspace,\zz)$. Dixmier and Douady
also proved that given a locally compact paracompact Hausdorff space
$\Tspace$, every element of $H^2(\Tspace,\torussheaf) \simeq
H^3(\Tspace,\zz)$ is the Dixmier-Douady invariant of some
continuous-trace C$^*$-algebra with spectrum $\Tspace$. This lead
naturally to the study of the Brauer group, which is the set of
spectrum preserving Morita equivalence classes of continuous-trace
class $C^*$-algebras with spectrum $\Tspace$ endowed with the
multiplication induced by spectrum preserving tensor product. The
map $\delta: \brauer(\Tspace) \to H^2(\Tspace,\torussheaf)$ is in
fact a group isomorphism (for a modern treatment of the subject, see
\cite{raeburn-williams}).

The original proof of the surjectivity of $\delta$ used the
contractibility of the unitary group of an infinite dimensional
Hilbert space as well as Zorn's lemma, and therefore it did not lead
to explicit constructions. Later I. Raeburn and J. Taylor \cite{rt}
provided a more direct approach. We shall henceforth take $\Tspace$
to be compact. Given a finite open cover $U$ of $\Tspace$ and a
cocycle $\nu \in Z^2(U,\torussheaf)$, Raeburn and Taylor constructed
an algebra $A(U,\nu)$ with spectrum $\Tspace$ and Dixmier-Douady
invariant represented by $\nu$. In addition, they observed that
there exists a locally compact Hausdorff \'etale groupoid $\RTg_U$
and a circle-valued groupoid 2-cocycle $\tau \in Z^2(\RTg_U,
\torus)$ such that $A(U,\nu) \simeq C^*(\RTg_U,\tau)$.

In this work we explore the asymptotic behavior of the algebras
$C^{\ast}(\RTg_U,\tau)$, as one refines the cover $U$ while
retaining a certain compatibility of cocycles. This is accomplished
in terms of an \'etale groupoid $G$ and an appropriate cocycle
$\sigma$. More precisely, we say that $\mathcal{W}$ is an
intersection refinement of $\mathcal{U}$, denoted $\mathcal{U} \leq
\mathcal{W}$, if there exists a cover $\mathcal{V}$ such that
$$\mathcal{W} = \mathcal{U} \cap \mathcal{V} = \{U_i \cap V_j~|~U_i
\in \mathcal{U},V_j \in \mathcal{V}\} .$$ Any sequence of open
covers $\{V^{(k)}\}_{k=1}^\infty$ gives rise to a sequence of
intersection refinements $$W^{(0)} \leq W^{(1)} \leq W^{(2)} \leq
... $$ of $\Tspace$, where $W^{(i+1)} = W^{(i)}\cap V^{(i+1)}$. We
denote $G_n = \RTg_{W^{(n)}}$, and from this sequence we construct
the groupoid $G$.

In order to carry out this construction, and in particular to define
the compatibility of the cocycles, one is lead naturally to
technical considerations. The naive compatibility condition, namely
that the sequence $\sigma_n \in Z^2(G_n,\torus)$ be generated from
pullbacks along the projection maps $G_{n+1} \to G_n$, is
insufficient, and a stronger condition is needed. In fact, we are
compelled to carry out a careful analysis of several associated
groupoids which are not \'etale, ultimately leading to the \'etale
groupoid $G$ and limit cocycle $\sigma$.

Our main result states the following (see Theorems \ref{thm:properties-G} and \ref{thm:main}).
Recall that the support of a set $S \subseteq C^*(G, \sigma)$ is the set of points in $G$
for which there is an element of $S$ that does not vanish as a function in $C_0(G)$.

\bigskip
\noindent \textbf{Main Theorem.} \textit{ Let $G_n$ be a sequence of
groupoids corresponding to a sequence of intersection refinements and $\sigma_n\in Z^2(G_n,
\torus)$ an associated sequence of compatible cocycles. There exists
a locally compact Hausdorff principal amenable \'etale groupoid $G$
and a cocycle $\sigma \in Z^2(G, \torus)$ as well as a sequence of
isometric $*$-homomorphisms $\varphi_n: C^*(G_n, \sigma_n) \to
C^*(G, \sigma)$ such that $G$ is the support of $\bigcup
\varphi_n(C^*(G_n, \sigma_n))$. }
\bigskip

In the particular case when $\Tspace$ is a point, $C^*(G)$ is a UHF
algebra (i.e. a direct limit of finite-dimensional \Cs), and
moreover all UHF algebras arise from an appropriate choice of
infinite intersection refinement. Furthermore, in that case we have
in fact that $C^*(G)$ is the direct limit of the algebras
$C^*(G_n)$.

In a more general setting the main issue is that we cannot construct
maps from $C^{\ast}(G_n,\sigma_n)$ to
$C^{\ast}(G_{n+1},\sigma_{n+1})$, and moreover  the set $S=\bigcup
\varphi_n(C^*(G_n, \sigma_n))$ is not always dense inside $C^*(G,
\sigma)$. Therefore we do not have a direct limit. However, $G$ is
the support of $S$, and we regard our construction as a generalized
direct limit, especially in light of the work of P. Muhly and B.
Solel \cite{muhly-solel-subalgs-groupoid-algs}.

Apart from presenting our main results, we study the properties of
the groupoid $G$ and several other groupoids related to covers and
to sequences of refinements.

\section{Preliminaries}

Let $G$ be a second countable locally compact Hausdorff principal
groupoid.  We shall denote the unit space of $G$ by $G^{(0)}$ and
the set of composable pairs by $G^{(2)}$. We shall also denote by
$r$ and $d$ the range and source maps, respectively, and set
$G^u=\{ x\in G : r(x)=u\}$ and $G_u=\{ x\in G : d(x)=u\}$ for all
$u\in G^{(0)}$. The groupoid $G$ is said to be
\textbf{r-discrete} if $G^{(0)}$ is open. We will say that $G$ is
an \textbf{\'etale} groupoid if $r$ is a local homeomorphism (We
remark that the notions of \'etale and r-discrete are a source of
confusion in the literature. We refer the reader to
\cite{resende} for a detailed discussion on the relations between
those definitions). If $G$ is an \etale groupoid, then the family
$\lambda= \{\lambda^u : u \in G^{(0)} \}$, where $\lambda^u$ is
counting measure on $G^u$, provides a continuous left Haar system
for $G$. We will always assume that \etale groupoids are endowed
with the counting Haar system.

We say that $H \subseteq G$ is a \textbf{subgroupoid} if it is
closed under the multiplication and inverse operations of $G$. We
emphasize that we do not require that $H$ and $G$ share the same
unit space.

Given two groupoids $G$ and $H$, a map $\phi: G \to H$ is a
\textbf{groupoid homomorphism} if for every composable pair
$(x,y)$, we have that $(\phi(x), \phi(y))$ is also composable,
$\phi(xy) = \phi(x) \phi(y)$, and for every $x\in G$,
$\phi(x^{-1}) = [\phi(x)]^{-1}.$

We shall say that $\sigma:G^{(2)}\to \mathbb{T}$ is a
\textbf{continuous 2-cocycle} if $\sigma$ is continuous and
$$
\sigma(x_0x_1,x_2) \sigma(x_0,x_1) = \sigma(x_1, x_2) \sigma(x_0,
x_1 x_2)
$$
for all $x_0, x_1, x_2 \in G$ such that $(x_0, x_1)$ and $(x_1,
x_2)$ are composable. We shall say that $\sigma:G^{(2)}\to
\mathbb{T}$ is a \textbf{continuous 2-coboundary} if there exists a
continuous function $\mu: G \to \mathbb{T}$ such that for all
$(x,y)\in G^{(2)}$,
$$
\sigma(x,y) = \mu(x) \mu(y) [\mu(xy)]^{-1}.
$$

The set $Z^2(G,\mathbb{T})$ of continuous 2-cocycles becomes a
group when it is endowed with the operation of pointwise
multiplication of $\mathbb{T}$-valued functions. The set
$B^2(G,\mathbb{T})$ of all continuous 2-coboundaries is a normal
subgroup, and we define the \textbf{second cohomology group} of
$G$ with coefficients in $\torus$ to be the quotient:
$$
H^2(G,\mathbb{T}) = Z^2(G,\mathbb{T}) / B^2(G,\mathbb{T}).
$$

Given a 2-cocycle $\sigma$, we define the $*$-algebra of functions
$C_c(G,\sigma)$ to be the set of functions $C_c(G)$ endowed with the
multiplication and involution
$$
(f \ast g)(x) = \int f(xy) g(y^{-1}) \sigma(xy, y^{-1})
d\lambda^{d(x)}(y),
$$
$$
f^*(x) = \overline{ f(x^{-1}) \sigma(x,x^{-1})}.
$$
For $f\in C_c(G,\sigma)$, set
$$
\| f \|_{I,r} = \sup_{u \in G^{(0)}} \int |f| d\lambda^u, \qquad \|
f \|_{I,d} = \sup_{u \in G^{(0)}} \int |f^*| d\lambda^u.
$$
We define a $*$-algebra norm $\| \cdot \|_I$ on $C_c(G,\sigma)$ by
the expression:
$$
\| f \|_I = \max\{ \| f \|_{I,r}, \| f \|_{I,d} \}.
$$
A \textbf{bounded representation} of $C_c(G,\sigma)$ on a Hilbert
space $H$ is a $*$-homomorphism  $\pi: C_c(G,\sigma) \to B(H)$ that
is nondegenerate, continuous (when $C_c(G,\sigma)$ is endowed with
the inductive limit topology and $B(H)$ has the weak operator
topology) and satisfies $\| \pi(f) \| \leq \| f\|_I$ for all $f \in
C_c(G,\sigma)$. The \textbf{full twisted groupoid $C^*$-algebra
$C^*(G,\sigma)$} is the closure of $C_c(G,\sigma)$ with respect to
the following $C^*$-norm:
$$
\| f \| = \sup \{ \| \pi(f) \|  : \pi \text{ is a bounded
representation of } C_c(G,\sigma) \}.
$$

Given a measure $\mu$ on $G^{(0)}$, we define a measure $\nu$ on $G$
by $\nu = \int_G \lambda^u d\mu(u)$ and write $\nu^{-1}$ for the
image of $\nu$ under inversion. Such a measure gives rise to a
bounded representation $\Ind_\mu$ of $C_c(G,\sigma)$ on
$L^2(\nu^{-1})$: for $f\in C_c(G,\sigma)$ and $\xi\in
L^2(\nu^{-1})$, set
$$
(\Ind_\mu(f)\xi)(x) =  \int_G f(xy) \xi(y^{-1}) \sigma(xy,y^{-1})
d\lambda^{d(x)}(y).
$$
The \textbf{reduced twisted groupoid $C^*$-algebra}
$C^*_{r}(G,\sigma)$ is the closure of $C_c(G,\sigma)$ with respect
to the following $C^*$-norm:
$$
\| f \|_{r} = \sup\{ \| \Ind_\mu(f)\| : \mu \text{ is a Radon
measure on } G^{(0)} \}.
$$
It is worth noting that by a disintegration argument (see
\cite{muhly-book-unpublished}):
$$
\| f \|_{r} = \sup\{ \| \Ind_\mu(f)\| : \mu \text{ is a unit point
mass on } G^{(0)} \}.
$$
In fact (see the comment in \cite{muhly-book-unpublished} following
the definition of the reduced norm), if $E$ is a saturating subset
of $G^{(0)}$ (i.e.  such that for every $u\in G^{(0)}$ there exists
$x\in G$ such that $r(x)=u$ and $d(x) \in E$), then:
\begin{equation}\label{eq:saturation-trick}
\| f \|_{r} = \sup\{ \| \Ind_\mu(f)\| : \mu \text{ is a unit point
mass on } E \}.
\end{equation}

We denote by $C_0(G)$ the Banach space of continuous functions on
$G$ which vanish at infinity. Renault (\cite{renault-book}, Proposition~II.4.2)
showed that if $G$ is an r-discrete groupoid with Haar system and
$\sigma$ a continuous 2-cocycle, then the injection $j:C_c(G) \to
C_0(G)$, extends to a norm-decreasing linear map
$j:C^*_{r}(G,\sigma) \to C_0(G)$ which is one-to-one. Therefore,
elements of $C^*_{r}(G,\sigma)$ can be viewed as continuous
functions on $G$.

We will say that a locally compact groupoid with continuous Haar
system $(G,\lambda)$ is \textbf{amenable} if there exists a net
$\{f_i\}$ in $C_c(G)$ such that:
\begin{description}
  \item[Am1] the functions $u \mapsto \int |f_i(y)|^2  d\lambda^{u}(y)$
  are uniformly bounded in the sup norm;

  \item[Am2] The functions $x \mapsto \int f_i(xy)\overline{f_i(y)}
  d\lambda^{d(x)}(y)$ converge to $1$ uniformly on any
  compact subset of $G$.
\end{description}

The property of amenability of $G$ guarantees that
$C^*_{r}(G,\sigma)$ is nuclear and moreover that $C^*_{r}(G,\sigma)$
and $C^*(G,\sigma)$ coincide.

Suppose that a topological groupoid $G$ is the union of an
increasing sequence of open subgroupoids $G_n$ all of which share
the unit space of $G$. In that case we will say that $G$ is the
inductive limit of the sequence $G_n$. When $G$ has a Haar
system, we assume that its restriction endows each element of the
sequence $G_n$ with a Haar system. The inductive limit of
amenable groupoids is amenable (see p.123 of \cite{renault-book}).

Let $G$ be a groupoid and let $X$ be a set. A \textbf{groupoid
action} of $G$ on $X$ (to the left) is given by a surjection
$r:X\rightarrow G^{(0)}$ and a map $(\gamma,x)\mapsto\gamma \cdot
x$ from $G*X:=\{ (\gamma,x) \ | \ d(\gamma)=r(x)\} $ to $X$,
satisfying
\begin{enumerate}
\item
$r(x) \cdot x=x$ for all $x\in X$.
\item
$r(\gamma \cdot x)=r(\gamma)$ for every  $(\gamma,x) \in G*X$.
\item
if  $(\gamma_1,x) \in G*X$ and  $(\gamma_2, \gamma_1) \in G^{(2)}$
then  $(\gamma_2\gamma_1,x), (\gamma_2,\gamma_1 \cdot x) \in G*X$
and $\gamma_2 \cdot (\gamma_1 \cdot x)= (\gamma_2\gamma_1) \cdot x$.
\end{enumerate}
We say that the (left) action is \textbf{free} when the equation
$\gamma \cdot x = x$ implies that $\gamma$ is a unit, namely
$\gamma = d(\gamma)=r(x)$. In the case where the set $X$ is
itself a groupoid, we will say that the action is
\textbf{transitive} if for every $x,y \in X$ such that
$d(x)=d(y)$, there exists $\gamma \in G$ such that $x=\gamma \cdot
y$. A groupoid homomorphism $\pi: X \rightarrow G$ is said to be
   \textbf{equivariant} with respect to the $G$-action if for every
   $(\gamma,x) \in G*X$ we have that $d(\gamma) = r(\pi(x))$ and
   $\pi(\gamma \cdot x) = \gamma \pi(x)$.

\section{Bundles of Glimm Groupoids}

Glimm groupoids are groupoids whose $C^*$-algebras are the well
known Glimm (or UHF) $C^*$-algebras. They serve as basic examples
in our study, however we present them here  in a slightly
non-standard fashion (compare with \cite{renault-book}).
Therefore we review some of their basic properties and we provide
proofs for which we could not find a convenient reference in the
literature.

Let $\nn$ denote the set $\{0,1,2,\dots\}$.

\begin{mydef}\label{def:admissible-set}
We shall say that a set $\Omega$ is \textbf{admissible} if there
exists $S \subseteq \nn$ of the form $S=\{0,1,2,\dots, n\}$ or
$S=\nn$ for which $\Omega \subseteq \prod_{k \in S} \nn$ and
furthermore $\Omega = \prod_{k\in S} \Omega_k$, where $\Omega_k
\subseteq \nn$ is finite and nonempty for every $k\in S$. We always
endow $\Omega$ with the (compact) product topology of the discrete
sets $\Omega_k$, $k\in S$.
\end{mydef}

It will be useful to regard the elements of $\Omega$
as sequences $\alpha=(\alpha_0,\alpha_1,\dots)$ or $(n+1)$-tuples
$\alpha=(\alpha_0,\alpha_1,\dots, \alpha_n)$.

We define an equivalence relation on an admissible set $\Omega$ as
follows. For $\alpha, \beta \in \Omega$, we say that
$$
\alpha \sim \beta \Leftrightarrow \alpha_k=\beta_k \text{ except for
finitely many values of } k.
$$
Notice that in the case when $\Omega$ is finite any two points are
equivalent. The aforementioned equivalence relation gives rise to a
principal groupoid
$$
\ruhf(\Omega) = \{ (\alpha, \beta) \in \Omega \times \Omega:  \alpha
\sim \beta \}
$$
where the pair $(\alpha, \beta)$, $(\gamma, \delta)$ is composable
if and only if $\beta=\gamma$, and $(\alpha, \beta)^{-1}=(\beta,
\alpha)$.

We define a topology on $\ruhf(\Omega)$ as follows. For every $(\alpha,
\beta)\in \ruhf(\Omega)$ and $N\in \nn$ such that $\alpha_n=\beta_n$ for
$n>N$, let
$$
\mathcal{O}_{(\alpha,\beta)}(N) = \{ (\gamma, \delta) \in
\ruhf(\Omega) : \forall k \leq N, \gamma_k=\alpha_k,
\delta_k=\beta_k \text{ and } \forall k> N, \gamma_k = \delta_k \}.
$$
The collection $\{ \mathcal{O}_\xi(N) : \xi \in \ruhf(\Omega), N \in \nn \}$ is
a basis for the topology of $\ruhf(\Omega)$. In terms of convergence of nets, a
net $(\alpha_\lambda, \beta_\lambda)$ in $\ruhf(\Omega)$ converges
to a point $(\alpha,\beta) \in \ruhf(\Omega)$ if and only if for
every $n\in\nn$ such that $\alpha_k=\beta_k$ for $k>n$ there exists
$\lambda_0$ such that for $\lambda\geq \lambda_0$,
$(\alpha_\lambda)_k=\alpha_k$ and $(\beta_\lambda)_k=\beta_k$ when
$k\leq n$ and $(\alpha_\lambda)_k=(\beta_\lambda)_k$ when $k>n$.

Note that the topology of $\ruhf(\Omega)$ has a countable basis
since
$\mathcal{O}_{(\alpha,\beta)}(N)=\mathcal{O}_{(\gamma,\delta)}(N)$
whenever $\alpha_k=\gamma_k$ and $\beta_k=\delta_k$ for $k\leq N$.

\begin{mydef} The principal topological groupoid $\ruhf(\Omega)$ will
be called the \textbf{Glimm groupoid of an admissible set
$\mathbf{\Omega}$}.
\end{mydef}

For example, when $\Omega$ is finite,
$\ruhf(\Omega)$ is the groupoid $\Omega \times \Omega$ corresponding to the
trivial equivalence relation on $\Omega$, and in particular it is compact. When
$\Omega$ is infinite this need not be the case. Nevertheless, its unit space is
still compact since the map $\alpha \mapsto (\alpha, \alpha)$ implements a
homeomorphism $\Omega \simeq \ruhf(\Omega)^{(0)}$.

\begin{prop} The Glimm groupoid of an admissible set $\Omega$ is a
locally compact Hausdorff groupoid which is principal and
\etale. Furthermore, the $C^*$-algebra $C^*(\ruhf(\Omega))$
is a Glimm algebra (also called uniformly hyperfinite or UHF).
\end{prop}
\begin{proof}
This is a straightforward verification, except for the statement
regarding the nature of the $C^*$-algebra,
which follows from Renault's study of Glimm groupoids and AF-groupoids in
\cite{renault-book} (see p.128 and results thereafter).
\end{proof}

\textbf{Let $\Tspace$ be a second countable compact Hausdorff space.}

\begin{mydef} The \textbf{bundle of Glimm groupoids of an admissible set
$\Omega$ over $\Tspace$}, which
will be denoted by $\Tspace \ast \ruhf(\Omega)$, is the set $\Tspace
\times \ruhf(\Omega)$ endowed with the pointwise groupoid structure
and the product topology. It will be convenient to work with
$\Tspace \ast \ruhf(\Omega)$ in the following presentation:
$$
\Tspace \ast \ruhf(\Omega) = \{ (\alpha, t, \beta) : t \in
\Tspace, (\alpha, \beta) \in \ruhf(\Omega) \}
$$
In this notation a pair $(\alpha, t, \beta)$, $(\gamma, s, \delta)$
is composable if and only if $\beta=\gamma$ and $t=s$, in which case
their product is $(\alpha, t, \delta)$, and $(\alpha, t,
\beta)^{-1}= (\beta, t, \alpha)$.
\end{mydef}

The groupoid $\Tspace \ast \ruhf(\Omega)$ arises naturally from the
equivalence relation on $\Tspace \times \Omega$ which is defined by
$(t,\alpha) \sim (s, \beta)$ if and only if $t=s$ and $\alpha \sim
\beta$. In particular, its unit space is given by
$\{(\alpha,t,\alpha) : t \in \Tspace, \alpha \in \Omega \} = \Tspace
\ast \ruhf^{(0)}(\Omega)$.

We collect many of the properties of $\Truhf{\Omega}$ in the following
proposition, although we omit its proof since it is a routine verification.
Recall that locally compact second countable Hausdorff spaces are metrizable by
Urysohn's metrization theorem (see \cite{kelley} pp. 125 and 147).

\begin{prop}\label{prop:ruhf} Given an admissible set $\Omega$ and  a
compact second countable Hausdorff space $\Tspace$, the groupoid
$\Tspace \ast \ruhf(\Omega)$ is locally compact, second
countable, Hausdorff, metrizable, principal and \etale.
Furthermore, its unit space is compact since it is naturally
homeomorphic to $\Tspace \times \Omega$.
\end{prop}

We now address the issue of amenability. When $\Omega$ is finite, as we have
remarked, the groupoid $\ruhf(\Omega)$ is compact. Therefore the \'etale
groupoid $\Truhf{\Omega}$ is compact, and it follows that it is amenable.

We now turn to the case when $\Omega$ is infinite and $\Omega =
\prod_{k\in\nn} \Omega_k$.

We introduce a bit more notation. For every $n\in\nn$, we will denote
$\Omega^{(n)}=\prod_{k=0}^n \Omega_k.$ Note that $\Omega^{(n)}$ is also an
admissible set.  Given $\alpha \in \Omega$ and $n\in\nn$, we shall denote
$\alpha|n = ( \alpha_0, \alpha_1, \dots,
\alpha_n )$. Notice that, using this notation, we have $\Omega^{(n)} =
\{ \alpha|n : \alpha \in \Omega \}$.

\begin{mydef} For each $n\geq 0$, denote by $\XX_n$ the subset of $\Truhf{\Omega}$
given by
$$
\XX_n= \{ (\alpha, t, \beta) \in  \Truhf{\Omega} : \alpha_k=\beta_k \text{ for }
k> n \}.
$$
\end{mydef}

Notice that the sets $\XX_n$ are nested: $\XX_n \subseteq
\XX_{n+1}$, and that $\Truhf{\Omega} =\bigcup_{n\in\nn} \XX_n$.

\begin{remark}\label{rem:conv-Truhf}
We can restate the convergence in $\Truhf{\Omega}$ as follows. We will denote by
$p_\Tspace: \Truhf{\Omega} \to \Tspace$ and $p_n:\Truhf{\Omega} \to
\ruhf(\Omega^{(n)})$ the canonical
projections given by $p_\Tspace(\alpha,t,\beta)=t$ and
$p_n(\alpha,t,\beta)=(\alpha|n, \beta|n)$. A net $\xi_\lambda$
converges to $\xi$ in $\Truhf{\Omega}$ if and only if
\begin{enumerate}
\item $p_\Tspace(\xi_\lambda) \to p_\Tspace(\xi)$.

\item $\forall n \in \nn$ such that $\xi \in \XX_n$, there exists $\lambda_0$
such that for any $\lambda \geq \lambda_0$, $\xi_\lambda \in \XX_n$ and
$p_n(\xi_\lambda) =p_n(\xi)$.
\end{enumerate}
We also note that the projections $p_\Tspace$ and $p_n$ are
continuous.
\end{remark}

\begin{lemma}\label{lemma:XX-properties}
For every $n\geq 0$ the set $\XX_n$ is a  compact open amenable \'etale
subgroupoid of $\Truhf{\Omega}$ such that $\XX_n^{(0)}=[\Truhf{\Omega}]^{(0)}$.
\end{lemma}
\begin{proof} It is clear that $\XX_n$ is a subgroupoid of $\Truhf{\Omega}$,
and $\XX_n$ contains  $[\Truhf{\Omega}]^{(0)}$,
hence $\XX_n^{(0)}=[\Truhf{\Omega}]^{(0)}$. It follows immediately from
Remark~\ref{rem:conv-Truhf} (2) that $\XX_n$ is open. In order to show that
$\XX_n$ is closed, suppose that
$\xi_\lambda \to \xi$ and $\xi_\lambda \in \XX_n$ for all $\lambda$.
Since $\xi \in \Truhf{\Omega} = \cup_{n\in\nn} \XX_n, \exists N\in \nn$ such
that $\xi \in \XX_N$. If $N\leq n$ we have that $\xi
\in \XX_N \subseteq \XX_n$ and in this case we are done. Now assume $N>n$. There
exists $\lambda_0$ such
that for $\lambda\geq \lambda_0$, $\xi_\lambda \in \XX_N$
and $p_N(\xi_\lambda) = p_N(\xi)$.  Thus, if
$\xi_\lambda=(\alpha_\lambda, t_\lambda, \beta_\lambda)$ and $\xi =
(\alpha, t, \beta)$, we have for $\lambda \geq \lambda_0$ that
$\alpha_k = (\alpha_\lambda)_k = (\beta_\lambda)_k = \beta_k$
for $k=n+1, \dots, N$. Hence, $\xi \in \XX_n$.

Let $Q_n= \Omega^{(n)} \times  \Omega^{(n)} \times
\prod_{k=n+1}^\infty \Omega_k$ endowed with the product topology,
with respect to which it is compact. Let $j: Q_n \to
\ruhf(\Omega)$ be the map given by $j(x,y,\omega)=(x\omega,
y\omega)$, where $x\omega$ denotes the concatenation of $x$ and
$\omega$. This map is clearly injective and it is continuous,
therefore $j(Q_n)$ is compact in $\ruhf(\Omega)$. Now observe that
$\XX_n \subseteq \Tspace \times j(Q_n)$, hence it is a closed
subset of a compact set of $\Truhf{\Omega}$ and we conclude that
it is compact.

The groupoid $\XX_n$ is \'etale because it is an open subgroupoid
of an \'etale groupoid. Since $r:\Truhf{\Omega} \to
[\Truhf{\Omega}]^{(0)} $ is a local homeomorphism, its
restriction to the open set $\XX_n$ is still a local
homeomorphism.

Finally, $\XX_n$ is amenable because it is compact.
\end{proof}

In summary, we have shown that $\Truhf{\Omega}$ is an inductive
limit of compact open amenable subgroupoids, and it follows in
particular that it is also amenable. We have proven the following
result.

\begin{prop}\label{prop:Truhf-amenable} Given an admissible set $\Omega$ and  a
compact second countable Hausdorff space $\Tspace$, the groupoid
$\Truhf{\Omega}$ is amenable.
\end{prop}

Observe that when $\Tspace=\{ \text{pt} \}$, we have that
$\Truhf{\Omega} \simeq \ruhf(\Omega)$. In particular, the latter
is amenable.

Although we chose an alternative path, it is possible to define
the topology of $\Truhf{\Omega}$ in terms of the inductive limit
of the sequence $\XX_n$ (see \cite{renault-book}, p. 122).

\section{Groupoids for Ordered Cover refinements}

\begin{mydef}
Let $\mathcal{U}$ and $\mathcal{W}$ be open covers of $\Tspace$.
We say that $\mathcal{W}$ is an \textbf{intersection refinement}
of $\mathcal{U}$ and denote it $\mathcal{U} \leq \mathcal{W}$ if
there exists an open cover $\mathcal{V}$ such that $\mathcal{W} =
\mathcal{U} \cap \mathcal{V} = \{U_i \cap V_j~|~U_i \in
\mathcal{U},V_j \in \mathcal{V}\}$.
\end{mydef}

It is easy to see that the set of intersection refinements is cofinal
in the set of all refinements.

Given any family of open covers $V = \{
V^{(0)},V^{(1)},V^{(2)}...\}$, we obtain a sequence of
intersection refinements $W^{(0)} \leq W^{(1)} \leq W^{(2)} \leq
... $ by defining $W^{(n)} = V^{(0)} \cap V^{(1)} \cap V^{(2)}
... \cap V^{(n)}$. We present this exact construction using the
notation and terminology we introduced in the previous section.

For each $k\in \nn$, let $V^{(k)} = \{ V^{(k)}_0, V^{(k)}_1, \dots
V^{(k)}_{n_k} \}$ be an open cover for $\mathcal{T}$. We emphasize
that we allow repetitions, i.e. $V^{(k)}_i = V^{(k)}_j$ for $i
\neq j$. The family $V = \{ V^{(0)},V^{(1)},V^{(2)}...\}$ is then
a family of open covers of $\mathcal{T}$. We can now apply the
terminology of the previous section. For each $k\in\nn$, denote
$\Omega_k=\{0, 1, \dots, n_k\}$ and let
$\Omega=\prod_{k=0}^\infty \Omega_k$ be the \textbf{admissible
set corresponding to $V$}. The set $\Omega$ is in fact admissible
in the sense of Definition~\ref{def:admissible-set}, as are the
sets $\Omega^{(n)}=\prod_{k=0}^n \Omega_k$, for every $n\in\nn$.

If $\alpha =(\alpha_0,\alpha_1,\alpha_2,\dots, \alpha_n) \in
\Omega^{(n)}$, we set:
$$
W_{\alpha} = \bigcap_{k=0}^n V^{(k)}_{\alpha_k}.
$$
Note that $W_{\alpha}$ is an open (possibly empty) set for every
$\alpha$, being a finite intersection of open sets. Moreover, for
every $n \in \nn$, $W^{(n)} = \{ W_{\alpha} : \alpha \in
\Omega^{(n)} \}$ is an open cover of $\mathcal{T}$, and $W^{(0)}
\leq W^{(1)} \leq W^{(2)} \leq ... $ is a sequence of
intersection refinements corresponding to $V$. In light of this,
we will call $V$ an \textbf{ordered cover refinement}.

When $\alpha =(\alpha_0,\alpha_1,\alpha_2,\dots) \in \Omega$, the
infinite intersection $ W_\alpha = \bigcap_{k=0}^\infty V^{(k)}_{\alpha_k}$ is
only a $G_\delta$ set. Thus, in this case $\{W_{\alpha} : \alpha \in \Omega \}$
is not an open cover of $\mathcal{T}$.

\begin{mydef}
For every $N\in\nn$, we define the groupoid
$$
G_N(V) = \{ (\alpha, t, \beta): \alpha,\beta \in \Omega^{(N)}, t \in
W_\alpha \cap W_\beta \}
$$
where a pair $(\alpha, t, \beta)$, $(\gamma, s, \delta)$ is
composable if and only if $\beta=\gamma$ and $t=s$, in which case
their product is $(\alpha, t, \delta)$, and $(\alpha, t,
\beta)^{-1}= (\beta, t, \alpha)$. The topology of $G_N(V)$ is the
relative topology from its natural inclusion into $\Tspace \ast
\ruhf(\Omega^{(N)})$.
\end{mydef}

The topology of $G_N(V)$  has as basis the collection of all
sets
$$
Z_{\alpha,\beta,U} = \{ (\alpha, t, \beta) \in G_N(V) : t \in U \}
$$
where $(\alpha, \beta) \in \ruhf(\Omega^{(N)})$ and $U\subseteq
\Tspace$ is open. The sets $Z_{\alpha,\beta, U}$ implicitly depend on $N$.

For every $N\in\nn$ the groupoid $G_N(V)$ arises naturally from
the restriction of the equivalence relation introduced on
$\Tspace \times \Omega^{(N)}$ to the set $\{ (t,\alpha) \in
\Tspace \times \Omega^{(N)} : t \in W_\alpha\}.$

The groupoid $G_N(V)$ is precisely the groupoid $\RTg(W^{(N)})$
corresponding to the Raeburn-Taylor \C of the finite open cover
$W^{(N)}$ of $\Tspace$. Presenting the elements in the form
$(\alpha,t,\beta)$, we keep track of all the covers
$V^{(0)},V^{(1)},V^{(2)},...,V^{(N)}$ from which $W^{(N)}$ was
obtained. The groupoid $G_N(V)$ has the following basic
properties.

\begin{prop}[\cite{rt}] For every $N\in\nn$ the groupoid $G_N(V)$ is
locally compact, Hausdorff, principal and \etale.
\end{prop}

\begin{mydef} For every $N\in\nn$, define the groupoid
$$
\gcheck_N(V) = \{ (\alpha, t, \beta): (\alpha,\beta) \in
\ruhf(\Omega^{(N)}), \; t \in \overline{W_{\alpha}} \cap
\overline{W_{\beta}} \}
$$
with the following operations: a pair $(\alpha, t, \beta)$,
$(\gamma, s, \delta)$ is composable if and only if $\beta=\gamma$
and $t=s$, in which case their product is $(\alpha, t, \delta)$, and
$(\alpha, t, \beta)^{-1}= (\beta, t, \alpha)$. The topology of
$\gcheck_N(V)$ is the relative topology from its natural inclusion
into $\Tspace \ast \ruhf(\Omega^{(N)})$.
\end{mydef}

The sets of the form
$$
\check{Z}_{\alpha, \beta, U} = \{ (\alpha, t, \beta) \in \gcheck_N(V) :
t\in U \}
$$
where $(\alpha, \beta) \in \ruhf(\Omega^{(N)})$ and $U \subseteq
\Tspace$ is open, constitute a basis for the topology of
$\gcheck_N(V)$.

\begin{prop}
For every $N\in\nn$, $\gcheck_N(V)$ is a compact Hausdorff
r-discrete principal groupoid. The embedding $J:G_N(V) \to
\gcheck_N(V)$ is a continuous open mapping and a groupoid
homomorphism.
\end{prop}

\begin{proof}
The groupoid $\gcheck_N(V)$ arises from the restriction of the
equivalence relation of $\Tspace \times \Omega$ to the set $\{
(t,\alpha) \in \Tspace \times \Omega^{(N)} : t \in
\overline{W_{\alpha}}\; \}$, hence it is a principal groupoid. We
claim that $\gcheck_N(V)$ is a closed subset of the compact
Hausdorff space $\Tspace\ast\ruhf(\Omega^{(N)})$ $\simeq \Tspace
\times \Omega^{(N)} \times \Omega^{(N)}$, hence compact and
Hausdorff in the relative topology. Let $(\alpha_\lambda,
t_\lambda, \beta_\lambda)$ be a net in $\gcheck_N(V)$ converging
to $(\alpha, t, \beta) \in \Truhf{\Omega^{(N)}}$. Then there
exists $\lambda_0$ such that for $\lambda\geq \lambda_0$,
$\alpha_\lambda=\alpha$ and $\beta_\lambda=\beta$, hence
$t_\lambda \in\overline{W_\alpha} \cap \overline{W_\beta}$. Since
$t_\lambda \to t$, it follows that $t\in \overline{W_\alpha} \cap
\overline{W_\beta}$ and $(\alpha, t, \beta) \in \gcheck_N(V)$.
Finally, the unit space $\gcheck_N^{(0)}(V)$ of $\gcheck_N(V)$ is
open, since $ \gcheck_N^{(0)}(V) = \bigcup_{\alpha \in
\Omega^{(N)}} \check{Z}_{\alpha, \alpha, \Tspace}.$ Hence
$\gcheck_N(V)$ is r-discrete.

It is clear that the mapping $J$ is a continuous groupoid
homomorphism. In order to prove that it is an open mapping, let $W
\subseteq G_N(V)$ be open and fix $(\alpha, t, \beta) \in W$. Now let
$U \subseteq W_\alpha \cap W_\beta$ be an open neighborhood of $t$
in $\Tspace$ such that $Z_{\alpha,\beta,U} \subseteq W$. We have
that if $(\gamma, s, \delta) \in \check{Z}_{\alpha, \beta, U}$, then
$\gamma=\alpha$, $\delta=\beta$ and $s\in U\subseteq W_\alpha \cap
W_\beta$, hence $(\gamma, s, \delta) \in G_N(V)$. In other words,
$\check{Z}_{\alpha,\beta,U} \subseteq G_N(V)$. Since it follows from
the definitions that $Z_{\alpha,\beta,U} = \check{Z}_{\alpha, \beta,
U} \cap G_N(V)$, we have that in this case $Z_{\alpha,\beta,U} =
\check{Z}_{\alpha, \beta, U}$, and therefore $\check{Z}_{\alpha,
\beta, U} \subseteq W$. It follows that $W$ is a union of open sets
of $\gcheck_N(V)$, thus open.
\end{proof}

Our goal is to ``take the limit as $N$ goes to $\infty$'' of the
Raeburn-Taylor groupoids $G_N(V)$. We are aiming for a groupoid
$G(V)$ which is locally compact, Hausdorff, principal and
\'etale. A first step in this direction is to consider the
following groupoid:

\begin{mydef} We shall refer to the topological groupoid
$$
G_\infty(V) = \{ (\alpha, t, \beta)~|~ (\alpha,\beta) \in
\ruhf(\Omega),\; t \in W_\alpha \cap W_\beta \}
$$
endowed with the following structure. A pair $(\alpha, t, \beta)$,
$(\gamma, s, \delta)$ is composable if and only if $\beta=\gamma$
and $t=s$, in which case their product is $(\alpha, t, \delta)$. The
inverse is given by $(\alpha, t, \beta)^{-1}= (\beta, t, \alpha)$.
The topology of $G_\infty(V)$ is the relative topology from its
natural inclusion in $\Tspace \ast \ruhf(\Omega)$.
\end{mydef}

It follows from the inclusion of $G_\infty(V)$ into $\Tspace \ast
\ruhf(\Omega)$ that $G_\infty(V)$ is in fact a groupoid. But in
general, this natural candidate for $G(V)$ fails to be locally
compact (see Example~\ref{ex:Ginfty-not-loc-compact}).
Furthermore, the topological closure of $G_\infty(V)$ in
$\Truhf{\Omega}$ need not be closed under the multiplication
induced from $\Truhf{\Omega}$ (see Example~\ref{rem:sharp}), and
thus it is not a groupoid. We are led to consider $\gcheckinf(V)$
which is the algebraic closure of $\overline{G_\infty(V)}$ inside
$\Truhf{\Omega}$. Although $\gcheckinf(V)$ is a locally compact
groupoid, it is also lacking because it need not be \'etale (see
Example~\ref{ex:gcheck-not-etale}). Ultimately we identify a
groupoid $G(V)$ which is a subgroupoid of $\gcheckinf(V)$
containing $G_\infty(V)$ and which satisfies all the above
desired properties. Having outlined our plan, we now provide the
definitions of $\gcheckinf(V)$ and $G(V)$, establish their
properties, and give some examples.

\begin{mydef} We endow the set
$$
\gcheckinf(V) = \{ (\alpha, t, \beta): (\alpha,\beta) \in \ruhf(\Omega), \; t
\in (\bigcap_{N=0}^\infty \overline{W_{\alpha|N}}) \cap
(\bigcap_{N=0}^\infty \overline{W_{\beta|N}}) \}
$$
with the following groupoid operations: a pair $(\alpha, t, \beta)$,
$(\gamma, s, \delta)$ is composable if and only if $\beta=\gamma$
and $t=s$, in which case their product is $(\alpha, t, \delta)$. The
inverse is given by $(\alpha, t, \beta)^{-1}= (\beta, t, \alpha)$.
The topology of $\gcheckinf(V)$ is the relative topology from its
natural inclusion into $\Tspace \ast \ruhf(\Omega)$. As we will see
shortly, $\gcheckinf(V)$ is a groupoid, and it is the algebraic
closure of $\overline{G_\infty(V)}$.
\end{mydef}

\begin{mydef} We endow the set
$$
G(V) = \{ (\alpha, t, \beta) \in \gcheckinf(V):  \exists n\in\nn, (\alpha|n, t,
\beta|n) \in G_n(V) \text{ and } \alpha_k = \beta_k \text{ for } k>n \}
$$
with the groupoid operations and the relative topology inherited
from its natural inclusion into $\gcheckinf(V)$. We will soon verify
that $G(V)$ is in fact a subgroupoid of $\gcheckinf(V)$.
\end{mydef}

In order to lighten the notation, henceforth we will fix $V$, the
sequence of intersection refinements, and we will refrain from
marking the dependence of the groupoids on the sequence. For
example, we will denote $\gcheckinf(V)$ by $\gcheckinf$.

Note that for any $N \in \mathbb{N}$ and $\alpha \in \Omega^{(N)}$ we have
that $\bigcap_{n=0}^N \overline{W_{\alpha|n}} =
\overline{W_{\alpha}}$. In light of this we view $\gcheckinf$ as
the counterpart of $\gcheck_N$ when replacing $N$ with $\infty$.

It may not be immediately apparent that $G$ is closed under the
operations of $\gcheckinf$. We address this issue in the following
proposition.

\begin{prop}\label{prop:Ginfty_and_badg_principal_and metrizable}
$\gcheckinf$ is a principal groupoid containing $G_\infty$ and $G$
as subgroupoids. In particular, $G$ is a principal groupoid.
\end{prop}

\begin{proof}
The operations on $\gcheckinf$ correspond to the equivalence
relation on $\Tspace \times \Omega$ corresponding to
$\Truhf{\Omega}$ but restricted to the set $ \{ (t,\alpha) \in
\Tspace \times \Omega ~|~ t \in \bigcap_{N=0}^\infty
\overline{W_{\alpha|N}}\; \}$. Therefore $\gcheckinf$ is a
principal groupoid. It is straightforward to check that
$G_\infty$ is a subgroupoid of $\gcheckinf$. Now suppose that
$(\alpha, t, \beta)$ and $(\delta, s, \gamma)$ are elements of
$G$ which are composable in $\gcheckinf$. Then we have that
$t=s$, $\beta=\delta$, and there are $n,m\in \nn$ such that $t
\in W_{\alpha|n} \cap W_{\beta|n}$, $t \in W_{\beta|m} \cap
W_{\gamma|m}$, $\alpha_k=\beta_k$ for $k>n$ and $\beta_j=
\gamma_j$ for $j>m$. Let us assume without loss of generality
that $n\leq  m$. We can then write:
\begin{align*}
 \alpha & = (\alpha_0, \alpha_1, \dots, \alpha_n, \alpha_{n+1}, \dots \alpha_m,
\alpha_{m+1}, \alpha_{m+2}, \dots)\\
 \beta  & = (\beta_0, \beta_1, \dots, \beta_n, \alpha_{n+1}, \dots \alpha_m,
\alpha_{m+1}, \alpha_{m+2}, \dots)\\
 \gamma &= (\gamma_0, \gamma_1, \dots, \gamma_n, \gamma_{n+1}, \dots ,
\gamma_m, \alpha_{m+1}, \alpha_{m+2}, \dots)
\end{align*}
Notice that since $t \in W_{\beta|m}$ and $t \in W_{\alpha|n}$, we
actually have that $t \in W_{\alpha|m}$. Since $t \in W_{\gamma|m}$,
we conclude that $(\alpha, t, \gamma) \in G$. It is clear that $G$
is closed under inverses, therefore it is a subgroupoid of
$\gcheckinf$.

Finally, $G_\infty$ and $G$ are principal groupoids because they are
subgroupoids of a principal groupoid.
\end{proof}

The following remark will prove useful.
\begin{remark}\label{r:top1}
Suppose $O$ is an open set, $A$ is a set and $x\in O\cap
\overline{A}$. Then $x\in \overline{O\cap A}$.

The proof is easy. Given a net $x_\lambda$ in $A$ converging to
$x$, there exists $\lambda_0$ such that for $\lambda \geq
\lambda_0$ we have $x_\lambda \in O$. Thus for $\lambda \geq
\lambda_0$ we have $x_\lambda \in O \cap A$ and $x_\lambda \to x$.
\end{remark}

\begin{prop}\label{prop:Ginfty-in-G}
The topological closure of $G_\infty$ in $\Truhf{\Omega}$ is given
by
\begin{equation}\label{eq:G-closure}
\overline{G_\infty} = \{ (\alpha, t, \beta) \in
\Tspace\ast\ruhf(\Omega) : t \in \cap_N \overline{ W_{\alpha|N}\cap
W_{\beta|N}} \}.
\end{equation}
Moreover, The groupoid $\gcheckinf$ is the algebraic closure of
$\overline{G_\infty}$ in $\Truhf{\Omega}$.
\end{prop}
\begin{proof} Suppose $(\alpha, t, \beta)\in \Truhf{\Omega}$ is the limit
    of a net $(\alpha_\lambda, t_\lambda, \beta_\lambda)_{\lambda \in
    \Lambda}$ in $G_\infty$. Fix $N\in \nn$. There exists
    $\lambda_0\in\Lambda$ such that for $\lambda \geq\lambda_0$,
    $\alpha_\lambda|N = \alpha|N$, $\beta_\lambda|N = \beta|N$. In
    particular, for
    $\lambda \geq \lambda_0$, $t_\lambda \in W_{\alpha_\lambda|N}\cap
    W_{\beta_\lambda|N} = W_{\alpha|N}\cap W_{\beta|N}$. Since
    $t_\lambda \to t$ we have that $t \in \overline{W_{\alpha|N}\cap
    W_{\beta|N}}$.

    Conversely, suppose $(\alpha, t,\beta) \in \Truhf{\Omega}$ and
    $t\in \cap_N \overline{ W_{\alpha|N}\cap W_{\beta|N}}$. Let $S$
    be a local basis of open neighborhoods of $t$, partially ordered
    by reverse inclusion, and let $\Lambda=\nn \times S$ be the
    directed set with the product partial ordering. Let $\lambda = (N,
    A)$ be fixed. Notice that $t \in \overline{W_{\alpha|N}\cap W_{\beta|N}}$,
    hence we can choose $t_\lambda \in A \cap W_{\alpha|N}\cap
    W_{\beta|N}$. Define $\alpha_\lambda|N=\alpha|N$,
    $\beta_\lambda|N=\beta|N$, and for $n>N$ set
    $(\alpha_\lambda)_n=(\beta_\lambda)_n$ to be any
    value so that $t_\lambda \in V^{(n)}_{(\alpha_\lambda)_n}$. This
    is possible because $V$ is an ordered cover refinement. Now it
    is clear that $(\alpha_\lambda, t_\lambda, \beta_\lambda)$ is a net in
    $G_\infty$ which converges to $(\alpha, t, \beta)$.

    Finally, we show that $\gcheckinf$ is the algebraic closure of
    $\overline{G_\infty}$ in $\Truhf{\Omega}$. It follows from the definition
    of $\gcheckinf$ in conjunction with \eqref{eq:G-closure} that
    $\overline{G_\infty} \subseteq \gcheckinf$. Therefore, in order
    to prove the statement, it suffices to show
    that for every $z \in \gcheckinf$ there exist $x, y \in
    \overline{G_\infty}$ such that $z=xy$.

    Fix $z=(\alpha, t, \beta) \in \gchecking$. There exists
    $M\in\nn$ such that $\alpha_n=\beta_n$ for $n>M$. Now pick any
    $\gamma \in \Omega$ such that  $t \in W_{\gamma|M}$ (again
    this is possible since $V$ is an ordered cover refinement) and
    $\gamma_n=\alpha_n$ for $n>M$. Then set $x=(\alpha, t, \gamma)$
    and $y=(\gamma, t, \beta)$.

    We claim that $x,y\in \overline{G_\infty}$. We prove it for $x$, since
    an analogous argument will yield the result for $y$. We must
    prove that $t \in \overline{W_{\alpha|N}\cap W_{\gamma|N}}$ for
    all $N\in \nn$. Fix $N\in \nn$, and recall that $t\in
    \overline{W_{\alpha|N}}$. Suppose first that $N\leq M$. Then
    $W_{\gamma|N}$ is open and contains $t$, hence by
    Remark~\ref{r:top1} we have $t \in \overline{W_{\alpha|N}\cap
    W_{\gamma|N}}$. When $N>M$  notice that
    $\overline{W_{\alpha|N}\cap
    W_{\gamma|N}}= \overline{W_{\alpha|N}\cap
    W_{\gamma|M}}$ and the latter contains $t$ by Remark~\ref{r:top1}.
    Finally, it is clear that $z=xy$.
\end{proof}

\begin{example}[\textit{In general $G_\infty \neq \overline {G_\infty}$ and
$\overline {G_\infty} \neq \gchecking$}]\label{rem:sharp}
Consider the ordered cover refinement $V$ given as follows: for
every $k$ let $V^{(k)}_0 = [-1,0)$, $V^{(k)}_1 = (0,1]$ and
$V^{(k)}_2 = [-1,1]$, and consider the following elements (we
denote the infinite repetition of a number by placing a bar over
it):
\begin{align*}
g = & (0000\overline{0}, 0, 0000\overline{0}) &
x = & (0222\overline{2}, 0, 2222\overline{2}) \\
y = & (2222\overline{2}, 0, 1222\overline{2}) & z =
&(0222\overline{2}, 0, 1222\overline{2})
\end{align*}

We have immediately that $G_\infty \neq \overline{G_\infty}$, since
we have a net $(000\overline{0}, -1/n, 0000\overline{0})$ in
$G_\infty$ converging to $g \not\in G_\infty$.

Observe that $z\in \gchecking$ but $z\not\in \overline{G_\infty}$ by
\eqref{eq:G-closure} since $\overline{W_{0} \cap W_{1}} =\emptyset$.
Incidentally, this shows already that $\overline{G_\infty}$ is not
closed under multiplication, for in that case it would be equal to
its algebraic closure $\gchecking$ since it is obviously closed
under inverses. More concretely, notice that $x, y \in
\overline{G_\infty}$ and $z=xy$, but $z\not\in \overline{G_\infty}$.

We remark that in general, one can verify that $\gcheckinf =
\overline{G_\infty} \cdot \overline{G_\infty}.$
\end{example}

We have seen that if $G$ is a topological groupoid and $S$ is a
subset closed under the groupoid operations of $G$, in general
$\overline{S}$ need not be closed under the operations of $G$. This
is different from the category of groups, where the closure of a
subgroup is automatically a subgroup. We point out that from the
axioms in the definition of a groupoid it follows that
$(\overline{S} \times \overline{S})\cap G^{(2)}$ determines the set
of composable pairs for $\overline{S}$. In particular,
$\overline{G_\infty}$ cannot be made into a groupoid with the
restriction of the operations of $\Truhf{\Omega}$ by attempting to
declare a smaller set of composable pairs.

\section{Properties of $\gchecking$ and $G$}

\begin{prop}\label{prop:gcheckinf-locompact}
$\gcheckinf$ is a closed subset of $\Truhf{\Omega}$, therefore it is
a metrizable locally compact groupoid.
\end{prop}
\begin{proof}
The topological space $\gcheckinf$ is metrizable because it is a
subspace of $\Truhf{\Omega}$ which is metrizable, by
Proposition~\ref{prop:ruhf}.

Let $x_i=(\alpha_i, t_i, \beta_i)$ be a sequence in $\gcheckinf$
converging to $x=(\alpha, t, \beta)$ in $\Truhf{\Omega}$, that is to
say, $t_i \to t$ in $\Tspace$ and $(\alpha_i, \beta_i) \to (\alpha,
\beta)$ in $\ruhf(\Omega)$. In order to show that $x \in
\gchecking$, we need only to show that $t \in (  \cap_{n\in\nn}
\overline{W_{\alpha|n}}) \cap ( \cap_{n\in\nn}
\overline{W_{\beta|n}})$. For each $k >0$, denote $F_k = (
\cap_{n=1}^k \overline{W_{\alpha|n}}) \cap ( \cap_{n=1}^k
\overline{W_{\beta|n}})$; this is a sequence of decreasing sets and
 of course we want to show that $t \in \cap_{k\in\nn} F_k$.
 Since $(\alpha_i, \beta_i) \to (\alpha, \beta)$ in $\ruhf(\Omega)$,
for every $k > 0$ there exists $i_k
>0$ such that for $i\geq i_k$, $(\alpha_i)_n = \alpha_n$ and
$(\beta_i)_n = \beta_n$ for $n\leq k$. Since $x_{i_k} \in
\gcheckinf$, we have that $t_{i_k} \in ( \cap_{n\in\nn}
\overline{W_{\alpha_{i_k}|n}}) \cap ( \cap_{n\in\nn}
\overline{W_{\beta_{i_k}|n}})$, therefore $t_{i_k} \in (
\cap_{n=1}^k \overline{W_{\alpha|n}}) \cap ( \cap_{n=1}^k
\overline{W_{\beta|n}}) = F_k$. Since $\{ F_k : k > 0 \}$ is a
decreasing sequence of closed sets and $t_{i_k} \to t$, we have that
$t \in F_k$ for every $k$. Therefore $x\in \gchecking$ and we have
shown that $\gchecking$ is a closed subset of $\Truhf{\Omega}$. In
particular, as a closed subset of a locally compact set,
$\gchecking$ is locally compact in its relative topology.
\end{proof}

The following subsets of $\gchecking$ will play an important role in
the sequel.

\begin{mydef} For each $n\geq 0$, define the subset of $\gcheckinf$ given by
$$
\YY_n= \gchecking \cap \XX_n =  \{ (\alpha, t, \beta) \in \gcheckinf
: \alpha_k=\beta_k \text{ for } k
> n \}.
$$
\end{mydef}

Notice that the sets $\YY_n$ are nested: $\YY_n \subseteq
\YY_{n+1}$, and that $\gcheckinf =\bigcup_{n\in\nn} \YY_n$.
Furthermore, the convergence in $\gchecking$ can be stated using the
sets $\YY_n$ instead of $\XX_n$ in Remark~\ref{rem:conv-Truhf}.

\begin{remark}\label{remark:YY-properties}
For every $n\geq 0$ the set $\YY_n$ is clearly a subgroupoid of
$\gcheckinf$. Furthermore it is a compact open subset of
$\gcheckinf$.  This follows immediately from
Lemma~\ref{lemma:XX-properties} and
Proposition~\ref{prop:gcheckinf-locompact}, since $\YY_n=
\gchecking \cap \XX_n$. We also point out that
$\YY_n^{(0)}=\gcheckinf^{(0)}$ since $\YY_n$ contains
$\gcheckinf^{(0)}$.
\end{remark}

\begin{mydef}\label{def:basis-G} For any open subset $U$ of $\Tspace$, $n
\in \nn$ and $(\alpha,\beta) \in \ruhf(\Omega^{(n)})$, define
$$
\zcheck_{\alpha,\beta, U, n} = \{ (\gamma, t, \delta) \in\gcheckinf : t\in U,
\gamma|n=\alpha, \delta|n=\beta, \text{ and } \forall k>n,
\gamma_k=\delta_k \}
$$
\end{mydef}

It is easy to check that the collection of all such sets forms a
basis for the topology of $\gcheckinf$. It is useful to notice
that this is in fact a basis of open G-sets, i.e. sets where $r$
and $d$ are injective. Therefore we also have that $G_\infty$ and
$G$ each have a basis of open G-sets in the relative topology.

\begin{theorem}\label{thm:properties gcheckinf}
The groupoid $\gcheckinf$ is locally compact, second countable,
Hausdorff, metrizable, principal and r-discrete.
\end{theorem}

\begin{proof}
We have already shown that $\gcheckinf$ is principal in
Proposition~\ref{prop:Ginfty_and_badg_principal_and metrizable},
and that it is metrizable and locally compact in
Proposition~\ref{prop:gcheckinf-locompact}. It also inherits from
$\Truhf{\Omega}$ the properties of being second countable and
Hausdorff.

Notice that $\gcheckinf^{(0)}$ is the set of all $(\alpha, t, \beta)\in
\gcheckinf$ such that $\alpha=\beta$. In order to prove the $\gcheckinf$
is r-discrete, we need to show that $\gcheckinf^{(0)}$ is open in
$\gcheckinf$. Let $(\alpha_\lambda, t_\lambda, \beta_\lambda)\in
\gcheckinf$ be a net converging to $(\alpha,t,\alpha)\in
\gcheckinf^{(0)}$. Then there exists $\lambda_0$ such that for $\lambda
\geq \lambda_0$, $(\alpha_{\lambda})_0=\alpha_0$,
$(\beta_{\lambda})_0=\alpha_0$ and
$\alpha_\lambda(n)=\beta_\lambda(n)$ for $n>0$. In other words,
$\alpha_\lambda=\beta_\lambda$ and hence $(\alpha_\lambda,
t_\lambda, \beta_\lambda)\in \gcheckinf^{(0)}$ for $\lambda\geq
\lambda_0$.
\end{proof}

\begin{example}[In general $\gcheckinf$ is not \etale.]\label{ex:gcheck-not-etale}

Let $\Tspace$ be the interval $[0,1]$. We define the ordered cover
refinement $V$ as follows:  let
$V^{(0)}_1=\Tspace$ ,  $V^{(0)}_2=(\frac{1}{2},1]$,
and for every $k >0$ let $V^{(k)}$ be the single set
$V^{(k)}_1=\Tspace$. In order to show that $\gcheckinf$ is not \'etale, it
suffices to
prove that the the range map $r : \gchecking \to \gchecking^{(0)}$ is not
open. Take the point $x=(111\overline{1},\frac{1}{2},211\overline{1}) \in
\gchecking$, and let $\mathcal{O}=\zcheck_{1,2,
(\frac{1}{4},\frac{3}{4}), 0} $ be an open
neighborhood of $x$. Any point $z \in \mathcal{O}$ must be of the form
$(111\overline{1},t,211\overline{1})$, where $t \in
[\frac{1}{2},\frac{3}{4})$. Therefore $r(\mathcal{O}) = \{
(111\overline{1},t,111\overline{1}) ~|~ t \in
[\frac{1}{2},\frac{3}{4}) \}, $ which is not an open set.
\end{example}

\begin{example}[In general $G_\infty$ is not locally
compact]\label{ex:Ginfty-not-loc-compact}

Let $\Tspace=[0,1]$, and for every $k \geq 0$ let $V^{(k)}_0=\Tspace$,
$V^{(k)}_1=(1/2,1]$. Recall that $\Truhf{\Omega}$ is metrizable by
Proposition~\ref{prop:ruhf}, hence
so are $G_\infty$ and $\gcheckinf$. We show that the point
$x=(000\overline{0},\frac{1}{2},000\overline{0})$,
does not have a compact neighborhood. In fact, if $\mathcal{O}$ is a compact
neighborhood of $x$, then it contains an open neighborhood of $x$ of the form
$\zcheck_{00\dots0,00\dots0, U, n} \cap G_\infty = \{
(\alpha, t, \alpha)\in
G_\infty : t\in U, \quad \alpha|n\equiv 0 \}$
for some $n \in \nn$ and $U$  an open subset of $\Tspace$.
Consider the sequence $x_k=(00\dots0111\overline{1},\; \frac{1}{2} + \frac{1}{k},\;
00\dots0111\overline{1})$ where there are exactly $n$ zeros, followed by
infinitely many ones.
For $k$ large enough, $x_k \in \zcheck_{00\dots0,00\dots0, U, n}\cap G_\infty$,
and $x_k$ converges in $\gcheckinf$ to $(00\dots0111\overline{1}, 1/2,
00\dots0111\overline{1})$, however this point is not in
$G_\infty$. It follows that this is a sequence in $\mathcal{O}$ without a
convergent subsequence in $G_\infty$.
\end{example}

\begin{lemma}\label{lem:pi}
The map $\pi_N: \gcheckinf \to \gcheck_N$ given by $\pi_N(\alpha,
t,\beta) = (\alpha|N, t, \beta|N)$ is a surjective continuous
mapping and a groupoid homomorphism.
\end{lemma}

\begin{proof}
It is easy to check that $\pi_N$ is a well-defined continuous
groupoid homomorphism. In order to show that it is surjective, let
$(x, t, y) \in \gcheck_N$. Since for every $n$, $V^{(n)}$ is a
cover of $\Tspace$, there exists a sequence $\gamma = \{\gamma_k
\} \in \Omega$ such that for any $k>N$, $V^{(k)}_{\gamma_k}$
contains $t$. Define $\alpha, \beta \in \Omega$ by setting
$\alpha|N=x$, $\beta|N=y$ and $\alpha_n=\beta_n=\gamma_n$ for
$n>N$. Clearly $(\alpha, \beta) \in \ruhf(\Omega)$. Moreover,
denote $O = V^{(N+1)}_{\gamma_{N+1}}$ and $A = W_{x}$. We then
have that $\overline{A} = \overline{W_x} = \bigcap_{n=0}^N
\overline{W_{\alpha|n}}$. Since $t \in O \cap \overline{A}$, we
have by Remark~\ref{r:top1}, that $t \in \overline{O \cap A} =
\overline{V^{(N+1)}_{\gamma_{N+1}} \cap W_{\alpha|N}} =
\overline{W_{\alpha|N+1}}$. Repeating this reasoning proves that
for every $M$, $ t \in \overline{W_{\alpha|M}} \cap
\overline{W_{\beta|M}}$. Thus $(\alpha, t, \beta) \in \gcheckinf$
and satisfies $\pi_N (\alpha, t, \beta) = (x, t, y)$.
\end{proof}

We omit the similar proof of the following lemma.
\begin{lemma}\label{lem:pi-n-m}
For every $n\leq m$, define a map $\pi_n^m:\gcheck_m \to \gcheck_n$
by $ \pi_n^m(\alpha,t,\beta) = (\alpha|n,t,\beta|n). $
The maps $\pi_n^m$ are surjective continuous groupoid homomorphisms.
\end{lemma}

\begin{lemma}\label{lem:G_open}
$G$ is an open subgroupoid of $\gcheckinf$.
\end{lemma}

\begin{proof}
We showed in Lemma~\ref{prop:Ginfty_and_badg_principal_and
metrizable} that $G$ is a subgroupoid of $\gcheckinf$. From the
definition of $G$ it follows that
$$
G = \bigcup_{n \in \nn} [\pi_n^{-1}(G_n) \cap \YY_n].
$$
For every $n\in \nn$, $G_n$ is open in $\gcheck_n$
and $\pi_n$ is continuous, hence $\pi^{-1}_n(G_n)$ is
an open subset of $\gcheckinf$. Since $\YY_n$ is an open set for all
$n$, we conclude that $G$ is an open subset of $\gcheckinf$.
\end{proof}

\begin{theorem}\label{thm:properties-G} $G$ is a locally compact, metrizable,
second countable, Hausdorff, principal groupoid. Furthermore, it
is \etale and amenable.
\end{theorem}
\begin{proof} As a subgroupoid of $\gcheckinf$, $G$ inherits the properties
of being second countable, Hausdorff and metrizable, and by
Proposition~\ref{prop:Ginfty_and_badg_principal_and metrizable}
it is principal. Furthermore, by Lemma~\ref{lem:G_open}, $G$ is
open in $\gcheckinf$ and the latter is locally compact, hence $G$
is locally compact.

In order to prove the $G$ is \etale, first we must show that
$G^{(0)}$ is open in $G$. This follows immediately from the
observation that $G^{(0)}=G \cap \gcheckinf^{(0)}$, and both are
open subsets of $\gcheckinf$ by Lemma~\ref{lem:G_open} and
since $\gcheckinf$ is r-discrete. Next we prove that $r$
is a local homeomorphism by showing that is is open. Since $G$ has
a basis of open G-sets, it follows that $r$ is a local homeomorphism
if and only if it is open.

Let $\mathcal{O}$ be a non-empty open subset of $G$, and let
$x=(\alpha, t, \beta) \in \mathcal{O}$. Let $Z_x=\check{Z}_{\alpha,
\beta, U, n} \cap G$ be an open neighborhood of $x$ inside
$\mathcal{O}$. Since $x\in G$, there exists $k \in \nn$ such that $t
\in W_{\alpha|k} \cap W_{\beta|k}$ and $\alpha_i = \beta_i$ for
$i>k$. We may assume without loss of generality that $n>k$, and also
that $U \subseteq W_{\alpha|k} \cap W_{\beta|k}$. We claim that
$r(Z_x) = \check{Z}_{\alpha, \alpha, U, n} \cap G$. It is clear that
$r(Z_x) \subseteq \check{Z}_{\alpha, \alpha, U, n} \cap G$. Fix
$(\gamma, s, \gamma) \in \check{Z}_{\alpha, \alpha, U, n} \cap G$.
Then we have that $s \in U$, $s \in \overline{W_{\gamma|j}}$ for all
$j$, and
$$
\gamma  = (\alpha_0, \alpha_1, \dots, \alpha_k, \alpha_{k+1},
\dots,\alpha_n, \gamma_{n+1}, \gamma_{n+2}, \dots).
$$
We define
$$
\delta  = (\beta_0, \beta_1, \dots, \beta_k, \alpha_{k+1},
\dots,\alpha_n, \gamma_{n+1}, \gamma_{n+2}, \dots),
$$
keeping in mind that $\alpha_{k+1}, \dots,\alpha_n = \beta_{k+1},
\dots,\beta_n$. Obviously $r(\gamma, s, \delta) = (\gamma, s,
\gamma)$. The fact that $(\gamma, s, \delta) \in \check{Z}_{\alpha,
\beta, U, n} \cap G=Z$ follows from the following observations:
\begin{enumerate}
\item
$s\in U \subseteq W_{\alpha|k} \cap W_{\beta|k} = W_{\gamma|k} \cap
W_{\delta|k}$, therefore $(\gamma|k,s,\delta|k) \in G_k$.
\item
$\gamma_i = \delta_i$ for all $i >k$.
\item
For all $j$, $s \in \overline{W_{\gamma|j}}$ .
\item
For all $j$, $s \in \overline{W_{\delta|j}}$ : for $j\leq k$ we have
that $s\in U \subseteq W_{\beta|j} = W_{\delta|j}$, and for $j> k$,
by Remark~\ref{r:top1}, we have that $s \in U \cap
\overline{W_{\gamma|j}} \subseteq \overline{U \cap W_{\gamma|j}}
\subseteq \overline{W_{\delta|j}}$ .
\end{enumerate}
We conclude that $r(Z_x) = \check{Z}_{\alpha, \alpha, U, n} \cap G $
is an open set, which is clearly contained in $ r(\mathcal{O})$. It
follows that $r(\mathcal{O})$ is open, since
$r(\mathcal{O})=\bigcup_x r(Z_x)$.

Finally, in order to prove amenability of $G$ we employ Proposition
5.1.1 of~\cite{renault-anantharaman-delaroche}, which states that if
$H$ is a locally closed subgroupoid of an amenable locally compact
groupoid $H'$ and the source and range maps of $H$ are open then $H$
is amenable. From Lemma~\ref{lem:G_open} we have that $G$ is open in
$\gcheckinf$, therefore there exists an open set $A$ in
$\Truhf{\Omega}$ such that $G = A \cap \gcheckinf$; since
$\gcheckinf$ is closed in $\Truhf{\Omega}$, we conclude that $G$ is
the intersection of open and closed subsets of $\Truhf{\Omega}$,
hence it is locally closed (see section I.3.3
of~\cite{bourbaki-topology1}). We have already seen that $G$ is
\'etale, therefore its range and source maps are open. By
Proposition~\ref{prop:Truhf-amenable}, $\Truhf{\Omega}$ is amenable,
hence we conclude that $G$ is amenable.
\end{proof}

\section{Maps between Groupoid C*-Algebras}

We present two propositions which are general, both of the same
nature: under certain assumptions on two groupoids, we obtain an
isometric $*$-homomorphism between the corresponding groupoid
$C^*$-algebras. The composition of the these maps will play a key
role in our main theorem.


\begin{prop}\label{prop:CQ-into-CH}
Let $\TG$ and $Q$ be locally compact Hausdorff principal \'etale
groupoids endowed with the respective counting Haar systems, and let
$\sigma\in Z^2(Q, \torus)$. Let $\pi: \TG \to Q$ be a surjective
continuous proper groupoid homomorphism, and suppose $Q$ acts on
$\TG$ to the left. Suppose the action is free, transitive and that
$\pi$ is equivariant with respect to the action. Then the map
$\pi^*:C_c(Q, \sigma) \to C_c(\TG, \pi^*\sigma)$ given by $\pi^*f =
f\circ \pi$ is an isometric $*$-homomorphism with respect to reduced
$C^*$-norms, which therefore extends to the reduced $C^*$-algebras.
\end{prop}

\begin{proof}

The map is well-defined because $\pi$ is continuous and proper,
hence given $f\in C_c(Q)$, the function $f\circ \pi$ is also
continuous and compactly supported on $\TG$.

It is straightforward to check that $\pi^*$ is linear and
$*$-preserving. We now show that it is multiplicative. Let $f, g \in
C_c(Q,\sigma)$ and fix $x\in \TG$.
\begin{align*}
(\pipull{f} * \pipull{g}) (x)
   & = \int_{\TG} \pipull{f}(xy) \pipull{g} (y^{-1}) \pipull{\sigma} (xy,y^{-1})
   d\lambda_\TG^{d(x)}(y)
\end{align*}
Recall that $\lambda^u$ is a measure supported on $\TG^u$ and
that $\pipull{\sigma} (x,y) = \sigma (\pi(x),\pi(y))$ by the
definition of a pullback cocycle. Thus, since $\pi$ is a groupoid
homomorphism, we obtain:
\begin{align*}
   (\pipull{f} * \pipull{g}) (x) =\int_{\TG^{d(x)}} f(\pi(x)\pi(y))
   g(\pi(y)^{-1}) \sigma (\pi(x)\pi(y),\pi(y)^{-1})
   d\lambda_\TG^{d(x)}(y)
\end{align*}

We state the next argument as a lemma.
\begin{lemma}
Suppose $u\in \TG^{(0)}$ and let $v=\pi(u)$. Then $\pi: \TG_u \to
Q_v$ is a bijection.
\end{lemma}

\begin{proof}
Take $x\in \TG_u$. By definition $d(x) = u = d(u)$. Therefore, since
the action is transitive, there exists $\gamma \in Q$ such that
$x=\gamma\cdot u$.

Since the action is free, $\gamma$ is determined uniquely. Indeed,
suppose $\gamma\cdot u = \widetilde{\gamma} \cdot u$. Then
$\gamma^{-1}\gamma\cdot u = \gamma^{-1}\widetilde{\gamma} \cdot u$,
so $d(\gamma) \cdot u = \gamma^{-1}\widetilde{\gamma} \cdot u$, and
since $d(\gamma) = r(u)$ and $r(u) \cdot u = u$ we see that $u =
\gamma^{-1}\widetilde{\gamma} \cdot u$. Freeness now implies that
$\gamma^{-1}\widetilde{\gamma} = d(\gamma^{-1}\widetilde{\gamma}) =
r(u) = d(\gamma)$. Thus $\gamma \gamma^{-1}\widetilde{\gamma} =
\gamma d(\gamma)= \gamma$, so $r(\gamma) \widetilde{\gamma} =
\gamma$. But $r(\gamma) = r(\gamma \cdot u) = r(\widetilde{\gamma}
\cdot u) = r(\widetilde{\gamma})$, hence $r(\widetilde{\gamma})
\widetilde{\gamma} = \gamma$, proving $\widetilde{\gamma}= \gamma$.

>From the equivariance of $\pi$ with respect to the action, it
follows that $$\pi(x) = \pi(\gamma \cdot u) = \gamma \pi(u) = \gamma
v .$$ In particular this shows that $d(\gamma) = v$, hence $\gamma
\in Q_v$ and $\gamma v = \gamma$. We conclude that $\pi(x) = \gamma$
for a unique $\gamma \in Q_v$, and therefore we have a bijection.
\end{proof}

We return to the proof of the proposition. By considering inverses,
we have a bijection $\pi: \TG^u \to Q^v$. So the map $\pi:\TG^{d(x)}
\to Q^{\pi(d(x))}$ is a bijection. Now $\pi(d(x)) = \pi(xx^{-1}) =
\pi(x) \pi(x^{-1}) = d(\pi(x))$, so $\pi:\TG^{d(x)} \to
Q^{d(\pi(x))}$ is a bijection and per force it is measure preserving
with respect to the counting Haar systems. It follows that:
\begin{align*}
   (\pipull{f} * \pipull{g}) (x) & =  \int_{Q^{d(\pi(x))}} f(\pi(x)z)
   g(z^{-1}) \sigma (\pi(x)z,z^{-1}) d\lambda_Q^{d(\pi(x))}(z)\\
    & =(f * g) (\pi(x)) \\
    & = \pipull{(f * g)} (x)
\end{align*}
Thus $\pi^*$ is multiplicative, and hence a $*$-homomorphism.

We prove next that $\pipull{}$ is an isometry.  Let $v\in \TG^{(0)}$
and let $\delta_v$ be the probability measure on $\TG^{(0)}$
concentrated on $v$. Let $\epsilon_v$ be the probability measure on
$Q^{(0)}$ concentrated on $\pi(v)$. We claim that the following
equality holds:
\begin{equation}\label{eq:isom-again}
\| \Ind^\TG_{\delta_v}(\pipull{f}) \| = \| \Ind^Q_{\epsilon_v}(f) \|
\end{equation}
By definition, $\Ind^\TG_{\delta_v}(\pipull{f})$ acts on $L^2(\TG,
\nu^{-1})$, where in this case $\nu=\int \lambda_\TG^u d\delta_v(u)
= \lambda_\TG^v$, hence $\nu^{-1}$ is counting measure on $\TG_v$
and $L^2(\TG, \nu^{-1})\simeq\ell^2(\TG_v)$. Similarly,
$\Ind^Q_{\epsilon_v}(f)$ acts on $\ell^2(Q_{\pi(v)})$.

The Haar system of $\TG$ is given by counting measures, therefore we
can write for $\xi \in \ell^2(\TG_v)$,
$$
[\Ind^\TG_{\delta_v}(\pipull{f})\xi](x) = \sum_{y\in \TG^v}
\pipull{f}(xy) \xi(y^{-1}) \pipull{\sigma}(xy,y^{-1})
$$
Since the map $\pi: \TG_v \to Q_{\pi(v)}$ is a bijection, we have a
unitary operator $W:\ell^2(Q_{\pi(v)}) \to \ell^2(\TG_v)$ given by
$W\psi= \psi\circ \pi$. We now show that
$$
    \Ind^\TG_{\delta_v}(\pipull{f}) W = W \Ind^Q_{\epsilon_v}(f).
$$
Indeed, if $x\in \TG_v$,
\begin{align*}
[\Ind^\TG_{\delta_v}(\pipull{f})W\psi](x) & = \sum_{y\in \TG^v}
\pipull{f}(xy) [W\psi](y^{-1})
\pipull{\sigma}(xy,y^{-1}) \\
& =\sum_{y\in \TG^v} f(\pi(x)\pi(y)) \psi(\pi(y)^{-1})
\sigma(\pi(x)\pi(y),\pi(y)^{-1}) \\
& = \sum_{z\in Q^{\pi(v)}} f(\pi(x)z) \psi(z^{-1})
\sigma(\pi(x)z,z^{-1}) \\
&= [\Ind^Q_{\epsilon_v}(f)\psi](\pi(x))\\
& = [W\Ind^Q_{\epsilon_v}(f)\psi](x)
\end{align*}
Therefore the operators $\Ind^\TG_{\delta_v}(\pipull{f})$ and
$\Ind^Q_{\epsilon_v}(f)$ are unitarily equivalent, so they share the
same norm. Thus we have obtained equation~\eqref{eq:isom-again}.

Finally, since every unit point mass on $Q$ is of the form
$\epsilon_v$ for some $v\in \TG$, we have:
\begin{align*}
\|\pipull{f}\|_{C^*_r(\TG,\pipull{\sigma}})
  &= \sup\{ \| \Ind^\TG_{\delta_v}(\pipull{f})\| : v\in \TG^{(0)} \} \\
  &= \sup\{ \| \Ind^Q_{\epsilon_v}(f)\| : v \in \TG^{(0)} \} \\
  &= \sup\{ \| \Ind^Q_\mu(f)\| : \mu \text{ is
a unit point mass on } Q^{(0)} \} \\
            &= \|f\|_{C^*_r(Q, \sigma)}
\end{align*}
\end{proof}


\begin{prop}\label{prop:amenable-inclusion}
Let $R$ be a locally compact, Hausdorff, \'etale amenable groupoid
admitting as left Haar system the counting measures system and let
$\sigma\in Z^2(R,\torus)$. Suppose that $\tg$ is an open
subgroupoid of $R$ endowed with the restriction Haar system, and
let $\rho: C_c(\tg,\sigma) \hookrightarrow C_c(R,\sigma)$ be the
natural inclusion obtained by extending functions as identically
zero outside $\tg$. Then $\rho$ is an isometric $*$-homomorphism,
and thus extends to the groupoid $C^*$-algebras.
\end{prop}

\begin{proof}

It is clear that the map $\rho$ is well-defined: let $f\in C_c(\tg)$ with
$K=\supp f$ compact. The set $K$ is closed in $R$ since $R$ is
Hausdorff, and $f$ is zero on $\partial K$. Since $\tg$ is open, the
extension to $R$ is continuous and compactly supported.

It is straightforward to check that $\rho$ is linear and
$*$-preserving. We now show that $\rho$ is multiplicative. Let $f, g
\in C_c(\tg)$. Since $\supp \rho(f) \subseteq \tg$ and $\supp \rho(g)
\subseteq \tg$, it follows that $\supp(\rho(f) \ast \rho(g)) \subseteq \tg$
because $\tg$ is closed under the groupoid operations of $R$. Thus
it suffices to verify that $[\rho(f)\ast \rho(g)](x) = [\rho(f \ast g)](x)$
for all $x\in \tg$. Fix $x\in \tg$.
\begin{align*}
[\rho(f) * \rho(g)] (x)
   & = \int_{R} \rho(f)(xy) \rho(g)(y^{-1}) \sigma (xy,y^{-1})
   d\lambda^{d(x)}(y)\\
   &= \int_\tg \rho(f)(xy) \rho(g)(y^{-1}) \sigma (xy,y^{-1})
   d\lambda^{d(x)}(y)\\
   & =\int_{\tg} f(xy) g(y^{-1}) \sigma (xy,y^{-1})
   d\lambda^{d(x)}(y)\\
   &= (f*g)(x)\\
   &= [\rho(f*g)](x)
\end{align*}
where we used the fact that the Haar system of $\tg$ is the
restriction of the Haar system of $R$.

We claim that $\rho$ is an isometry:
$$
\| \rho(f) \|_{C^*(R,\sigma)} = \| f \|_{C^*(\tg,\sigma)}, \qquad
\forall f \in C_c(\tg,\sigma) .
$$
Observe that bounded representations of $C_c(R,\sigma)$ give rise to
bounded representations of $C_c(\tg,\sigma)$ by composition with
$\rho$ and taking cut-downs to ensure non-degeneracy. Thus we have
$$
 \| \rho(f) \|_{C^*(R,\sigma)} \leq \| f \|_{C^*(\tg,\sigma)}, \qquad
\forall f \in C_c(\tg,\sigma).
$$

In order to prove the opposite inequality, recall that $\tg$ is an
open subgroupoid of $R$ amenable, hence it is also amenable (see
Proposition 5.1.1 in \cite{renault-anantharaman-delaroche}). Thus
it suffices to convert to the reduced norms, and show that
$$
\| \rho(f) \|_{C_r^*(R,\sigma)} \geq \| f \|_{C_r^*(\tg,\sigma)},
\qquad \forall f \in C_c(\tg,\sigma).
$$
So let $\delta$ be a probability measure on $\tg^{(0)}$ concentrated
on a unit $v$. We will abuse notation slightly and denote also by
$\delta$ the probability measure on $R^{(0)}$ supported on $v$. By
definition, $\Ind^R_\delta(f)$ acts on $L^2(R,\nu^{-1})$, where in
this case $\nu=\int \lambda^u d\delta(u) = \lambda^v$, hence
$\nu^{-1}$ is counting measure on $R_v$ and
$L^2(R,\nu^{-1})\simeq\ell^2(R_v)$. Furthermore, we can write for
$\psi \in \ell^2(R_v)$,
$$
[\Ind_\delta^R(\rho(f))\psi](x) = \sum_{y\in R^v} \rho(f)(xy)
\psi(y^{-1}) \sigma(xy,y^{-1}).
$$
The operator $\Ind^\tg_{\delta}(f)$ acts analogously on
$\ell^2(\tg_v)$.

Consider the inclusion operator $W:\ell^2(\tg_v) \to \ell^2(R_v)$
where $W\xi$ is the extension of $\xi$ as identically zero outside
of $\tg_v$. Then $W$ is an isometry and $W^*:\ell^2(R_v) \to
\ell^2(\tg_v)$ is the restriction operator $W^*\psi =
\psi\downharpoonright_{\tg_v}$. Now observe that given $\psi \in
\ell^2(\tg_v)$ and $x\in \tg_v$,
\begin{align*}
(\Ind^R_\delta(\rho(f))W \psi)(x) & = \sum_{y\in R^v} \rho(f)(xy)
W\psi(y^{-1}) \sigma(xy,y^{-1}) \\
& = \sum_{y\in \tg^v} \rho(f)(xy)
\psi(y^{-1}) \sigma(xy,y^{-1}) \\
& = \sum_{y\in \tg^v} f(xy)
\psi(y^{-1}) \sigma(xy,y^{-1}) \\
& = (\Ind^\tg_\delta(f)\psi)(x) \\
& = W(\Ind^\tg_\delta(f)\psi)(x)
\end{align*}
If $x \not\in \tg_v$ then $xy \not\in \tg$, so $\rho(f)(xy) =0$ and
the whole expression is zero. Thus for any $\psi \in \ell^2(\tg_v)$
we have
$$ (\Ind^R_\delta(\rho(f))W \psi) = W(\Ind^\tg_\delta(f)\psi) .$$
Since $W$ corresponds to extension with zero outside of $\tg_v$ and
the norm is the $\ell^2$ norm, we conclude that:
$$
\| \Ind^R_\delta(\rho(f))W \psi \| = \| W \Ind^\tg_\delta(f)\psi \|
= \| \Ind^\tg_\delta(f)\psi \|
$$
This allows us to obtain the following inequality:
\begin{align*}\label{aha}
\| \Ind^R_\delta(\rho(f)) \| &= \sup \{ \|
\Ind^R_\delta(\rho(f))\psi \|  : \psi \in \ell^2(R_v), \| \psi\|
\leq 1 \} \\
& \geq \sup \{ \| \Ind^R_\delta(\rho(f))W \psi \|  : \psi \in
\ell^2(\tg_v), \| \psi\|
\leq 1 \} \\
& = \sup \{ \| \Ind^\tg_\delta(f) \psi \|  : \psi \in \ell^2(\tg_v),
\| \psi\|
\leq 1 \} \\
& = \| \Ind^\tg_\delta(f) \|
\end{align*}
Therefore we have:
\begin{align*}
\|\rho(f)\|_{C^*_r(R, \sigma)} & = \sup\{ \|
\Ind^R_\delta(\rho(f))\| : \delta \text{ is
a unit point mass on } R^{(0)} \} \\
& \geq \sup\{ \| \Ind^R_\delta(\rho(f))\| : \delta \text{ is
a unit point mass on } \tg^{(0)} \} \\
            &\geq  \sup\{ \| \Ind^\tg_\delta(f)\| : \delta \text{ is
a unit point mass on } \tg^{(0)} \} \\
            &= \|f\|_{C^*_r(\tg, \sigma)}
\end{align*}
This completes the proof.
\end{proof}

\section{The $C^*$-algebra of $G$}

Our main theorem asserts that under a certain assumption regarding
the cocycles, for every $n$ there is an isometric $*$-homomorphism
$$\varphi_n: C^*_r(G_n, \sigma_n) \to C^*(G, \sigma).$$ In order to
state our theorem precisely, we require the following definition.

\begin{mydef}
Given a cocycle $\sigma_n \in Z^2(G_n, \torus)$, we will say that it
extends to $\gcheck_n$ if there exists $\check{\sigma_n} \in
Z^2(\gcheck_n, \torus)$ such that $\sigma_n = \check{\sigma}_n$ on
$G_n^{(2)}$. Observe that such an extension is unique if it exists.
We denote by $Z_{\text{ext}}^2(G_n,\torus)$ the subgroup of cocycles
which can be extended to $\gcheck_n$.
\end{mydef}

A continuous 2-cocycle $\sigma_n \in Z^2(G_n, \torus)$ which extends
to $\gcheck_n$, will also extend to a continuous 2-cocycle $\sigma
\in Z^2(G, \torus)$, by means of pullback: Recall the maps $\pi_N:
\gcheckinf \to \gcheck_N$ given by $\pi_N(\alpha, t,\beta) = (\alpha|N,
t, \beta|N)$, introduced in Lemma~\ref{lem:pi}. When
restricted to $G$, the maps $\pi_N: G \to \gcheck_N$ remain
continuous groupoid homomorphisms, although they are no longer
surjective. Nevertheless, if $\check{\sigma}_n \in Z^2(\gcheck_n, \torus)$
then $\sigma=\pi_n^*\check{\sigma}_n$ is a continuous 2-cocycle in
$Z^2(G, \torus)$.

We can now state our main theorem.
\begin{theorem}\label{thm:main}
For every $n \geq 0$ and $\sigma_n\in Z^2_\text{ext}(G_n, \torus)$,
denote by $\check{\sigma}_n$ its extension in $Z^2(\gcheck_n,
\torus)$. Let $\sigma=\pi_n^*\check{\sigma}_n \in Z^2(G, \torus)$ be
the pullback cocycle obtained from the map $\pi_n: G \to \gcheck_n$.
Then the map $\varphi_n:C_c(G_n, \sigma_n) \to C_c(G, \sigma)$ given
by
$$
[\varphi_n(f)](x) = \left\{
  \begin{array}{ll}
    f(\pi_n(x)) \ \ \ &  x\in \pi_n^{-1}(G_n)\cap \YY_n \\
    0 \ \ \ & \hbox{otherwise}
  \end{array}
\right.
$$
is an isometric $*$-homomorphism, which extends to an isometric
$*$-homomorphism $\varphi_n: C_r^*(G_n, \sigma_n) \to C^*(G,
\sigma)$.
\end{theorem}

\begin{proof} Fix $n\in \nn$. We will apply Propositions \ref{prop:CQ-into-CH}
and \ref{prop:amenable-inclusion} in succession.

Let $Q=G_n$, $\tg=\pi_n^{-1}(G_n)\cap \YY_n$, and $R=G$. Notice that
$\tg \subseteq G$ and consider the following maps:
\begin{itemize}
\item
$\pi: \tg \to G_n$ given by the restriction of $\pi_n$ to $\tg$, and
its pullback
$$
\pi^*:C_c(G_n, \sigma_n) \to C_c(\tg, \pi^*\sigma_n)=C_c(\tg,
\sigma).
$$
\item $\rho:C_c(\tg, \sigma) \to C_c(G, \sigma)$ which extends functions
as identically zero outside of $\tg$.
\end{itemize}

We start with the map $\rho$. Notice first that $\tg$ is an open
subgroupoid of $G$. We will endow $\tg$ with the restriction
counting Haar system. $G$ satisfies all the conditions of
Proposition~\ref{prop:amenable-inclusion} by
Theorem~\ref{thm:properties-G}. Therefore the map $\rho$ is an
isometric $*$-homomorphism.

We now turn to the map $\pi^*$. Inheriting properties from $G$ of
which it is an open subgroupoid, $\tg$ as a groupoid is \'etale,
principal, locally compact and Hausdorff.

The map $\pi$ is clearly a surjective continuous groupoid
homomorphism. To see that it is proper, let $K \subseteq G_n$ be
compact. It will be useful to refer not only to the projection
$\pi_n: \tg \to G_n$, but also to the original on $\gcheckinf$, which
we temporarily denote by $\check{\pi}_n:\gcheckinf \to \gcheck_n$ for
distinction. Notice that
$$
\pi^{-1}(K) =  \pi_n^{-1}(K) \cap \YY_n = \check{\pi}_n^{-1}(K) \cap
\YY_n.
$$
$K$ is compact hence closed in $\gcheck_n$, since $\gcheck_n$ is
Hausdorff. Thus $\check{\pi}_n^{-1}(K)$ is closed in $\gcheckinf$
since $\check{\pi}_n$ is continuous by Lemma~\ref{lem:pi}. Since
$\YY_n$ is compact in $\gcheckinf$, it follows that $\pi^{-1}(K)$
is a compact set.

Let $G_n$ act on $\tg$ on the left with respect to the map
$r_{act}:\tg \to G_n^{(0)}$ given by $r_{act}(x)=\pi(r_\tg(x))$ as
follows. An element $\gamma=(i,t,j) \in G_n$ can act on $x=(\alpha,
s, \beta) \in \tg$ if and only if $d_{G_n}(\gamma)=r_{act}(x)$, that
is to say $t=s$ and $j=\alpha|n$. In that case,
$$
\gamma \cdot x := ( i_0 i_1 \dots i_n \alpha_{n+1} \alpha_{n+2}
\dots, t, \beta)
$$
It is clear that the map $r_{act}$ is a surjection. The action map
$(\gamma,x) \mapsto \gamma \cdot x$ is well-defined: $\gamma
=(i,t,j)$ is in $G_n$, so $t \in W_i\cap W_j$. Since $x=(\alpha, t,
\beta) \in \tg$, we have that $\alpha_k=\beta_k$ for $k>n$, $t \in
W_{\beta|n}$ and $t \in \overline{W_{\beta|k}}$ for all $k$. Thus by
Remark~\ref{r:top1} we have that $t\in
\overline{W_i \cap W_{\beta|k}}$ for all $k$. In particular, if we
set $\delta=i_0i_1\dots i_n\beta_{n+1}\dots$, then we have that for
every $k$, $t\in \overline{W_{\delta|k}}$ . We conclude that the
element $\gamma\cdot x = (\delta, t, \beta)$ indeed belongs to
$\tg$.

It is a straightforward verification that this is a left action,
i.e. that axioms (1) - (3) of the definition are satisfied. This
action is free because if $\gamma \cdot x = x$ and $x=(\alpha, t,
\beta)$ then we must have $\gamma=(\alpha|n, t, \alpha|n)=
d_{G_n}(\gamma) = r_{act}(x)$. In order to see that it is
transitive, suppose $y\in \tg$ and $d_Z(y)=d_Z(x)$. Then we can
write $y=(j_1j_2\dots j_n\beta_{n+1}\beta_{n+2}\dots, t, \beta)$.
Since both $x,y \in \tg$, we have that $t \in W_{\alpha|n}$ and
$t\in W_j$. Therefore, if we set $\gamma = (\alpha|n, t, j)$ we
have that $\gamma$ is an element of $G_n$ and $\gamma \cdot y =
x$. Last but not least, $\pi$ is equivariant with respect to the
action.   Indeed, if $(\gamma, x) \in G_n*\tg$ then
$d_{G_n}(\gamma) = r_{act}(x) = \pi(r_{\tg}(x))$. It is
straightforward to verify that $\pi(r_{\tg}(x)) =
r_{G_n}(\pi(x))$, therefore $d_{G_n}(\gamma) = r_{G_n}(\pi(x))$.
It is obvious that $\pi(\gamma \cdot x) = \gamma \pi(x)$.

Having verified all the conditions of Proposition~\ref{prop:CQ-into-CH},
we conclude that $\pi^*$ is an isometric $*$-homomorphism.

Finally, we observe that $\varphi_n$ is precisely $\rho \ \circ \
\pi^*$. Thus $\varphi_n:C_c(G_n, \sigma_n) \to C_c(G, \sigma)$ is an
isometric $*$-homomorphism. It follows from the general theory of
$C^*$-algebras that $\varphi_n$ extends to the $C^*$-completions.
\end{proof}

\section{$C^*(G,\sigma)$ as a Generalized Direct Limit}

We first consider the most trivial case where $\Tspace$
is a single point. Open covers of $\Tspace$ are merely repetitions
of the singleton set, and are determined by their cardinality. It is
easy to see that in this case $G$, $G_\infty$ and $\gcheckinf$ all
coincide with each other and with $\Tspace \ast \ruhf(\Omega)$,
which is in turn isomorphic to $\ruhf(\Omega)$. Therefore $C^*(G)$ is a UHF
$C^*$-algebra (The cohomology of $\Tspace$ is of course trivial, so
there are no cocycles involved). Moreover, the algebras
$C^*(G_n)$ in this case are matrix algebras, and
$\displaystyle{C^*(G) = \lim_{\rightarrow} C^*(G_n)}$.

We would also have $C^*(G)$ as a direct limit in the case where
$\Tspace$ is a finite (discrete) set of points, where $C^*(G)$
can be seen to be an AF $C^*$-algebra. However, in general this
is not the case. In order to be precise, we require the following
definition.
\begin{mydef}\label{def:compatible_cocycles}
Let $n \geq 0$ and take $\sigma_n\in Z^2_\text{ext}(G_n, \torus)$.
Denote by $\check{\sigma}_n$ its extension in $Z^2(\gcheck_n,
\torus)$. Let $m > n$, denote by $\check{\sigma}_{m} \in
Z^2(\gcheck_{m}, \torus)$ the pullback cocycle obtained from the map
$\pi_n^{m}: \gcheck_{m} \to \gcheck_n$ (see Lemma
\ref{lem:pi-n-m}), and let $\sigma_{m} \in Z^2_\text{ext}(G_{m},
\torus)$ be its restriction to $G_{m}$. We will then say that
$\sigma_n$ and $\sigma_{m}$ are \textbf{compatible cocycles}.

Now fix $n_0 \geq 0$. Notice that if $n\leq m \leq k$ then
$\pi_n^k = \pi_n^m \circ \pi_m^k$ and furthermore, $\pi_n =
\pi_n^m \circ \pi_m$. It follows that we can extend this notion
to define a sequence of compatible cocycles $\{ \sigma_n\}_{n
\geq n_0}$. Moreover, there is a well-defined cocycle $\sigma \in
Z^2(G, \torus)$ which is simultaneously the pullback of all the
cocycles $\{ \check{\sigma}_n\}_{n \geq n_0}$ with respect to the
maps $\pi_n: G \to \gcheck_n$. We will say that $\{ \sigma_n\}_{n
\geq n_0}$ is a \textbf{sequence of compatible cocycles with a
limit cocycle $\sigma$}.
\end{mydef}

Let $\{ \sigma_n\}_{n \geq n_0}$ be a sequence of compatible
cocycles with a limit cocycle $\sigma$. In general,
$C^*(G,\sigma)$ is \textit{not} a direct limit of the sequence
$C^*(G_n,\sigma_n)$ since we have no maps $C^*(G_n,\sigma_n) \to
C^*(G_{n+1},\sigma_{n+1})$. Moreover,
$\bigcup_n(\varphi_n(C^*(G_n,\sigma_n)))$ is not dense inside
$C^*(G,\sigma)$.

To see this, consider the following two points in $G$:
$x=(\alpha,t,\alpha)$ and $y=(\beta,t,\beta)$, where $\alpha =
\alpha_0 \alpha_1 \dots \alpha_n \alpha_{n+1} \alpha_{n+2} \dots$
and $\beta = \alpha_0 \alpha_1 \dots \alpha_n \beta_{n+1}
\beta_{n+2}\dots$, and where $t \in W_{\alpha|n}=W_{\beta|n}$ but
$t \not\in W_{\alpha|n+1}$ and $t \not\in W_{\beta|n+1}$. Thus
$x,y \in \pi_n^{-1}(G_n)\cap \YY_n$ whereas $x,y \not\in
\pi_{m}^{-1}(G_n)\cap \YY_m$ for $m>n$. Therefore
$\varphi_k(f)(x) = f(\pi_k(x)) = f(\pi_k(y)) = \varphi_k(f)(y)$
for any $k \leq n$ and any $f \in C_c(G_k,\sigma_k)$, and
$\varphi_k(f)(x) = 0 = \varphi_k(f)(y)$ for any $k > n$ and any
$f \in C_c(G_k,\sigma_k)$. Thus
$\bigcup_n(\varphi_n(C^*(G_n,\sigma_n)))$ cannot separate $x$ and
$y$.

Despite this, we regard $C^*(G,\sigma)$ as a generalized direct
limit of the sequence $C^*(G_n,\sigma_n)$. Our justification for
this is that $G$ satisfies the following minimality property.
Given a subset $S$ of $C^*(G,\sigma)$, we will denote
$$
\supp S  = \{ x\in G ~|~ \exists f \in S \text{ such that }
f(x)\neq 0 \}.
$$

\begin{prop}\label{prop:support}
 $G=\supp \bigcup_n(\varphi_n(C^*(G_n,\sigma_n)))$.
\end{prop}

\begin{proof}
We claim that for every $x \in G$ there exists $n \in \mathbb{N}$
and $f \in C_c(G_n,\sigma_n)$ such that $\varphi_n(f)(x) \neq 0$.
The proof of this is simple. Let $x=(\alpha,t,\beta) \in G$. There
exists $n \in \mathbb{N}$ such that $x \in \pi_n^{-1}(G_n)\cap
\YY_n$, i.e. $\pi_n(x) \in G_n$. Take $f \in C_c(G_n,\sigma_n)$ such
that $f(\pi_n(x)) \neq 0$. Then clearly $\varphi_n(f)(x) \neq 0$.
\end{proof}

\begin{remark}
In \cite{muhly-solel-subalgs-groupoid-algs}, P. Muhly and B. Solel
present a bijective correspondence between closed subsets of $G$ and
$C^*(G^{(0)})$-bimodules. In order to invoke their results, certain
assumptions on the groupoid $G$ are required. Our $G$ has all the
required properties, except one: $G$ does not admit a cover by
\textit{compact} open G-sets. (Note that $\gcheckinf$ does admit
such a cover - take the basis sets of the form
$\check{Z}_{\alpha,\beta, \Tspace, n}$. These are compact since they
can be written as $p_n^{-1}( \{(\alpha|n, \beta|n)\}) \cap \YY_n$.)
Nevertheless, it may be that the results of Muhly and Solel remain
valid without the compact open G-sets assumption. Should this be the
case, we would have as a corollary of Proposition \ref{prop:support}
the following statement: The $C^*(G^{(0)})$-bimodule generated by
$\bigcup_n(\varphi_n(C^*(G_n,\sigma_n)))$ is $C^*(G,\sigma)$.
\end{remark}

\section*{Acknowledgments}

We thank Alex Kumjian, Paul Muhly, Iain Raeburn, Jean Renault and
Pedro Resende for useful remarks and discussions. This research is
based in part on the first named author's Ph.D dissertation,
supervised by Eli Aljadeff and Baruch Solel, and we extend to them
our appreciation and many thanks. A. C. was partly supported by the
Graduate School of the Technion - Israel Institute of Technology,
and D. M. was partially supported by a Technion Swiss Society
Postdoctoral Fellowship.

\bibliographystyle{amsalpha}
\bibliography{groupoids}

\end{document}